\documentclass[12pt]{amsart}
\usepackage{verbatim,amssymb,amsmath,graphicx,hyperref,nicefrac,bbm,caption}
\usepackage[usenames]{color}

  \scrollmode

\newcommand{\R}{{\mathbb R}}
\newcommand{\N}{{\mathbb N}}
\newcommand{\Z}{{\mathbb Z}}
\newcommand{\EE}{{\mathbb E}}
\newcommand{\PP}{{\mathbb P}}
\newcommand{\A}{\mathcal A}
\newcommand{\1}{\mathbbm{1}}

\theoremstyle{plain}
\newtheorem{theorem}{Theorem}

\newtheorem{lemma}{Lemma}
\newtheorem{cor}{Corollary}

\theoremstyle{definition}
\newtheorem{rem}{Remark}
\newtheorem{ex}{Example}

\begin{document}
\title[]{On stochastic differential equations with arbitrary
slow convergence rates for strong approximation}

\author[Jentzen]
{Arnulf Jentzen}
\address{
Seminar f\"ur Angewandte Mathematik\\
Departement Mathematik\\
HG G 58.1\\
R\"amistrasse 101 \\
8092 Z\"urich\\
Switzerland} \email{arnulf.jentzen@sam.math.ethz.ch}

\author[M\"uller-Gronbach]
{Thomas M\"uller-Gronbach}
\address{
Fakult\"at f\"ur Informatik und Mathematik\\
Universit\"at Passau\\
Innstrasse 33 \\
94032 Passau\\
Germany} \email{thomas.mueller-gronbach@uni-passau.de}

\author[Yaroslavtseva]
{Larisa Yaroslavtseva}
\address{
Fakult\"at f\"ur Informatik und Mathematik\\
Universit\"at Passau\\
Innstrasse 33 \\
94032 Passau\\
Germany} \email{larisa.yaroslavtseva@uni-passau.de}

\begin{abstract}
In the recent article [Hairer, M., Hutzenthaler, M., \& Jentzen, A., Loss of regularity for Kolmogorov
equations, \emph{Ann.\ Probab.} {\bf 43} (2015), no. 2, 468--527] it has been shown that there exist stochastic
differential equations (SDEs) with infinitely often differentiable and globally bounded coefficients
such that the Euler scheme converges to the solution in the strong sense but with no polynomial rate.
Hairer et al.'s result naturally leads to
the question whether this slow convergence
phenomenon can be
overcome by using a
more sophisticated approximation method
than the simple Euler scheme.
In this article we answer this question to the negative.
We prove that there exist
SDEs with infinitely often differentiable and globally bounded coefficients 
such that no approximation method based on finitely many observations of the driving Brownian motion converges in absolute mean to the solution with a polynomial rate.
Even worse, we prove that 
for every arbitrarily slow convergence speed 
there exist SDEs with infinitely
often differentiable and globally bounded coefficients 
such that no approximation method based on finitely many observations of the driving Brownian motion can converge in absolute mean to the solution faster than the given speed of convergence. 
\end{abstract}

\maketitle

\section{Introduction}

Recently, it has been shown in Theorem~5.1 in Hairer et al.~\cite{hhj12} that there exist stochastic
differential equations (SDEs) with infinitely often differentiable and globally bounded coefficients such that
the Euler scheme converges 
to the solution but with no polynomial rate, neither in the strong sense nor in
the numerically weak sense.
In particular, Hairer et al.'s work~\cite{hhj12} includes the following result as a
special case.
\begin{theorem}[Slow convergence of the Euler scheme]
\label{thm:intro1}
Let $ T \in (0,\infty) $, $ d \in \{ 4, 5, \dots \} $,
$ \xi \in \R^d $.
Then there exist infinitely often differentiable and globally
bounded functions
$ \mu, \sigma \colon \R^d\to \R^d$
such that for every probability space
$ ( \Omega, \mathcal{F}, \PP ) $,
every normal filtration $ ( \mathcal{F}_t )_{ t \in [0,T] } $ on $ ( \Omega, \mathcal{F}, \PP ) $,
every standard $ ( \mathcal{F}_t )_{ t \in [0,T] } $-Brownian motion
$
  W \colon [0,T] \times \Omega \to \R
$
on $ ( \Omega, \mathcal{F}, \PP ) $,
every continuous $ ( \mathcal{F}_t )_{ t \in [0,T] } $-adapted stochastic process $ X \colon [0,T] \times \Omega \to \R^d $
with
$
  \forall \, t \in [0,T] \colon
  \PP\big(
  X(t)
  = \xi
  + \int_0^t \mu\big( X(s) \big) \, ds
  + \int_0^t \sigma\big( X(s) \big) \, dW(s)
  \big) = 1
$,
every sequence of mappings
$ Y^n \colon \{ 0, 1, \dots, n \} \times \Omega \to \R^d $,
$ n \in \N $,
with
$
  \forall \, n \in \N, k \in \{ 0, 1, \dots, n \}
  \colon
  Y^n_k =
  \xi + \sum_{ l = 0 }^{ k - 1 }
  \big[
    \mu\big(
      Y^n_l     
    \big)
    \frac{ T }{ n }
    +
    \sigma\big(
      Y^n_l
    \big)
    \big(
      W(( l + 1 ) T / n )
      -
      W( l T / n )
    \big)
  \big]
$,
and every $ \alpha \in (0,\infty) $
we have
\begin{equation}
  \lim_{ n \to \infty }
  \big(
  n^{ \alpha }
  \cdot
  \EE\big[
    \| X( T ) - Y^n_n \|
  \big]
  \big)
  =
  \infty
  .
\end{equation}
\end{theorem}
Theorem \ref{thm:intro1} 
naturally leads to
the question whether this slow convergence phenomenon can be
overcome by using a 
more sophisticated approximation method
than the simple Euler scheme.
Indeed, the literature on approximation of SDEs contains a number
of results on approximation schemes that are specifically designed for non-Lipschitz coefficients and in fact achieve
polynomial strong convergence rates for suitable classes of such SDEs
(see, e.g.,
\cite{h96,hms02,Schurz2006,MaoSzpruch2013Rate,HutzenthalerJentzenKloeden2012,
WangGan2013,Sabanis2013ECP,Sabanis2013Arxiv,Beynetal2014,TretyakovZhang2013}
for SDEs with monotone coefficients
and see, e.g., 
\cite{BerkaouiBossyDiop2008,GyoengyRasonyi2011,DereichNeuenkirchSzpruch2012,
Alfonsi2013,NeuenkirchSzpruch2014,HutzenthalerJentzen2014,
HutzenthalerJentzenNoll2014CIR,ChassagneuxJacquierMihaylov2014}
for SDEs with possibly non-monotone coefficients)
and one might hope that one of these schemes is able to overcome the slow 
convergence phenomenon stated in Theorem~\ref{thm:intro1}.
In this article we destroy this hope by answering the question posed above to the negative. We prove that there exist
SDEs with infinitely often differentiable and globally bounded coefficients 
such that no approximation method based on finitely many 
observations of the driving Brownian 
motion 
(see \eqref{eq:intro2} for details)
converges in absolute mean to the 
solution with a polynomial rate.
This fact 
is the subject of the next theorem, which immediately follows
from Corollary~\ref{cor1b} in Section~\ref{strong}.
\begin{theorem}
\label{thm:intro2}
Let $ T \in (0,\infty) $, $ d \in \{ 4, 5, \dots \} $,
$ \xi \in \R^d $.
Then there exist infinitely often differentiable and globally
bounded functions
$ \mu, \sigma\colon \R^d\to \R^d$
such that for every probability space
$ ( \Omega, \mathcal{F}, \PP ) $,
every normal filtration $ ( \mathcal{F}_t )_{ t \in [0,T] } $ on $ ( \Omega, \mathcal{F}, \PP ) $,
every standard $ ( \mathcal{F}_t )_{ t \in [0,T] } $-Brownian motion
$
  W \colon [0,T] \times \Omega \to \R
$
on $ ( \Omega, \mathcal{F}, \PP ) $,
every continuous 
$ ( \mathcal{F}_t )_{ t \in [0,T] } $-adapted 
stochastic process $ X \colon [0,T] \times \Omega \to \R^d $
with
$
  \forall \, t \in [0,T] \colon
  \PP\big(
  X(t)
  = \xi
  + \int_0^t \mu\big( X(s) \big) \, ds
  + \int_0^t \sigma\big( X(s) \big) \, dW(s)
  \big) = 1
$,
and every $ \alpha \in (0,\infty) $
we have
\begin{equation}
\label{eq:intro2}
  \lim_{ n \to \infty }
  \Bigl(
  n^{ \alpha }
  \cdot
  \inf_{ s_1, \dots, s_n \in [0,T] }
  \inf_{
    \substack{
      u \colon \R^n \to \R
    \\
      \text{measurable}
    }
  }
  \EE\Big[
    \big\|
      X( T )
      -
      u\big( W( s_1 ), \dots, W( s_n )
      \big)
    \big\|
  \Big]
  \Bigr)
  =
  \infty
  .
\end{equation}
\end{theorem}
Even worse,
our next result states
that for every arbitrarily slow convergence speed 
there exist SDEs with infinitely
often differentiable and globally bounded coefficients 
such that no approximation method 
that uses finitely many observations and, additionally, starting from some positive time,  the whole path of the driving Brownian motion, 
can converge in absolute mean to the solution faster than the given speed of convergence. 
\begin{theorem}
\label{thm:intro3}
Let $ T \in (0,\infty) $, $ d \in \{ 4, 5, \dots \} $,
$ \xi \in \R^d $
and let
$ ( a_n )_{ n \in \N } \subset (0,\infty) $ and 
$
  (\delta_n)_{ n \in \N }
  \subset (0,\infty)
$ be sequences of strictly 
positive reals such that
$
  \lim_{ n \to \infty } a_n = \lim_{ n \to \infty } \delta_n = 0
$.
Then there exist infinitely often differentiable and globally
bounded functions
$ \mu, \sigma \colon \R^d\to \R^d$
such that for every probability space
$ ( \Omega, \mathcal{F}, \PP ) $,
every normal filtration $ ( \mathcal{F}_t )_{ t \in [0,T] } $ on $ ( \Omega, \mathcal{F}, \PP ) $,
every standard $ ( \mathcal{F}_t )_{ t \in [0,T] } $-Brownian motion
$
  W \colon [0,T] \times \Omega \to \R
$
on $ ( \Omega, \mathcal{F}, \PP ) $,
every continuous $ ( \mathcal{F}_t )_{ t \in [0,T] } $-adapted stochastic process $ X \colon [0,T] \times \Omega \to \R^d $
with
$
  \forall \, t \in [0,T] \colon
  \PP\big(
  X(t)
  = \xi
  + \int_0^t \mu\big( X(s) \big) \, ds
  + \int_0^t \sigma\big( X(s) \big) \, dW(s)
  \big) = 1
$,
and every $ n \in \N $
we have
\begin{equation}
\label{eq:intro3}
  \inf_{ s_1, \dots, s_n \in [0,T] }
  \,
  \inf_{
    \substack{
      u \colon \R^n \times C( [ \delta_n, T ] ) \to \R  
    \\
      \text{measurable}
    }
  }
  \EE\Big[
    \big\|
      X( T )
      -
      u\big(
        W( s_1 ), \dots, W( s_n )
        ,
        ( W(s) )_{ s \in [ \delta_n , T ] }
      \big)
    \big\|
  \Big]
  \geq
  a_n
  .
\end{equation}
\end{theorem}
Theorem~\ref{thm:intro3} is an immediate consequence of Corollary~\ref{cor4}
    in Section~\ref{strong} together with
an appropriate scaling argument.
Roughly speaking, such SDEs can 
not be solved approximately in the strong sense in a reasonable computational time
as long as approximation methods based on finitely many evaluations of the driving
Brownian motion are used.
In Section~\ref{sec:numerics} we illustrate Theorems~\ref{thm:intro2} and \ref{thm:intro3} 
by a numerical example.

Next we point out that our results do neither cover the class of strong approximation
algorithms that may use finitely many arbitrary linear functionals of the 
driving Brownian motion nor cover strong approximation algorithms that may 
choose the number as well as the location of the evaluation nodes for the 
driving Brownian motion in a path dependent way.
Both issues will be the subject of future research.

We add that for strong approximation of SDEs with globally Lipschitz coefficients there is a multitude of results 
on lower error bounds already available in the literature; 
see, e.g.,~\cite{ClarkCameron1980,hmr01,m02,MG02_habil,m04,mr08,Ruemelin1982}, 
and the references therein. 
We also add that Theorem~2.4 in Gy\"{o}ngy~\cite{g98b} establishes, as a special case, the almost sure convergence rate
$ \nicefrac{ 1 }{ 2 } - $ for the Euler scheme and SDEs with globally bounded and infinitely often differentiable coefficients.
In particular, we note that there exist SDEs with globally bounded and infinitely often differentiable coefficients which, roughly speaking, 
can not be solved approximatively in the strong sense in a reaonsable computational time 
(according to Theorem~\ref{thm:intro3} above) but might be solveable, approximatively, in the
almost sure sense in a reasonable computational time (according to Theorem~2.4 in Gy\"{o}ngy~\cite{g98b}).

\section{Notation}

Throughout this article the following notation is used.
For a set $ A $, a vector space $ V $, a set $ B \subseteq V $,
and a function $ f \colon A \to B $
we 
put
$
  \operatorname{supp}( f ) = \left\{ x \in A \colon f(x) \neq 0 \right\}
$.
Moreover, for a natural number $ d \in \N $
and a vector $ v \in \R^d $
we 
denote by
$
  \|v\|_{ \R^d }
$
the Euclidean norm of $ v \in \R^d $.
Furthermore,
for a real number $ x \in \R $
we put
$
  \lfloor x \rfloor = \max\!\left(
    \Z \cap ( - \infty, x ]
  \right)
$
and
$
  \lceil x \rceil = \min\!\left(
    \Z \cap [ x, \infty )
  \right)
$.

\section{A family of stochastic differential equations with smooth and globally bounded coefficients}
\label{sec:setting}

Throughout this article we study SDEs provided by the following setting.

Let
$ T  \in (0,\infty) $,
let
$ ( \Omega, \mathcal{F}, \PP ) $
be a probability space with a normal filtration
$ ( \mathcal{F}_t )_{ t \in [0,T] } $,
and let
$
  W \colon [0,T] \times \Omega \to \R
$
be a 
standard $ ( \mathcal{F}_t )_{ t \in [0,T] } $-Brownian motion
on $ ( \Omega, \mathcal{F}, \PP ) $.
Let
$ \tau_1, \tau_2, \tau_3 \in \R $
satisfy
$
  0 < \tau_1 \leq \tau_2 < \tau_3 < T
$
and let
$ f, g, h \in C^{ \infty }( \R, \R ) $
be globally bounded and satisfy
$
  \operatorname{supp}( f ) \subseteq ( - \infty, \tau_1 ]
$,
$
  \inf_{ s \in [ 0, \nicefrac{ \tau_1 }{ 2 } ] } | f'(s) | > 0
$,
$
  \operatorname{supp}( g ) \subseteq [ \tau_2, \tau_3 ]
$,
$
 \int_{ \R } \left| g(s) \right|^2 ds > 0
$,
$
  \operatorname{supp}( h ) \subseteq [ \tau_3, \infty )
$,
and
$
  \int_{ \tau_3 }^{ T } h(s) \, ds \neq 0
$.

For every 
$
  \psi \in C^{ \infty }( \R , (0,\infty) )
$
let
$
  \mu^{\psi} \colon \R^4 \to \R^4
$
and 
$
  \sigma \colon \R^4 \to \R^4
$
be the functions such that for all 
$
  x = ( x_1, \dots, x_4 ) \in \R^4
$
we have
\begin{equation}
\label{coeff}
\begin{aligned}
  \mu^{ \psi }(x) & = \bigl( 1, 0, 0, h( x_1 ) \cdot \cos( x_2 \, \psi( x_3 ) ) \bigr)
\qquad
  \text{and}
\qquad
  \sigma(x) =
  \bigl(
    0, f( x_1 ) , g( x_1 ), 0
  \bigr)
\end{aligned}
\end{equation}
and let
$ X^{ \psi } = ( X^{ \psi }_1, \dots, X^{ \psi }_4 ) \colon [0,T] \times \Omega \to \R^4 $
be an $ ( \mathcal{F}_t )_{ t \in [0,T] }  $-adapted continuous
stochastic processes with the property that for all
$ t \in [0,T] $
it holds $ \PP $-a.s.\ that
$
  X^{ \psi }( t )
  =
  \int_0^t \mu^{ \psi }( X^{ \psi }( s ) ) \, ds + \int_0^t \sigma( X^{ \psi }( s ) ) \, dW(s)
$.

\begin{rem}
Note that for  all $ \psi \in C^{ \infty }(  \R, (0,\infty) ) $
we have that
$ \mu^{ \psi } $
and 
$\sigma$ are infinitely often differentiable and globally bounded.
\end{rem}

\begin{rem}

Note that for all
$ \psi \in C^{ \infty }( \R , ( 0, \infty ) ) $,
$ t \in [0,T] $
it holds $ \PP $-a.s.\ that
\begin{equation}\label{sol}
\begin{split}
  X_1^{ \psi }(t) & = t ,
\qquad
  X_2^{ \psi }(t) = \int_0^{ \min\{ t, \tau_1 \} } f(s) \, dW(s),
\\
  X_3^{ \psi }(t) & = \mathbbm{1}_{ [ \tau_2 , \, T ] }( t ) \cdot
  \int_{ \min\{ t, \tau_2 \} }^{ \min\{ t , \tau_3 \} } g(s) \, dW(s) ,
\\
  X_4^{ \psi }(t) & =
  \mathbbm{1}_{ [ \tau_3 , \, T ] }(t) \cdot
  \cos\bigl(
    X_2^{ \psi }( \tau_1 ) \,
    \psi\big(
      X_3^{ \psi }( \tau_3 )
    \big)
  \bigr)
  \cdot
  \int_{ \tau_3 }^{ t } h(s) \, ds
  .
\end{split}
\end{equation}
\end{rem}

\begin{ex}
\label{ex:fgh}
Let 
$ c_1, c_2, c_3 \in \R $
and let 
$ f, g, h \colon \R \to \R $ be the functions such that 
for all $ x \in \R $
we have
\begin{equation}
\begin{split}
  f(x) & = \1_{(-\infty,\tau_1)}(x)\cdot
    \exp\Bigl(
      c_1 +
      \frac{
        1
      }{
        x - \tau_1
      }
    \Bigr)
  ,
\\
  g(x) & = \1_{(\tau_2,\tau_3)}(x)\cdot
    \exp\Bigl(
      c_2 +
      \frac{ 1 }{ \tau_2 - x } + \frac{ 1 }{ x - \tau_3 }
    \Bigr)
  ,
\\
  h(x) & =
    \1_{(\tau_3,\infty)}(x)\cdot
    \exp\Bigl(
      c_3 
      +
      \frac{ 1 }{ \tau_3 - x }
    \Bigr)
    .
\end{split}
\end{equation}
Then $ f, g, h $ satisfy the conditions stated above,
that is,
$ f, g, h $ are 
infinitely often differentiable
and
globally bounded
and 
$ f, g, h $ satisfy
$
  \operatorname{supp}( f ) \subseteq ( - \infty, \tau_1 ]
$,
$
  \inf_{ s \in [ 0, \nicefrac{ \tau_1 }{ 2 } ] } | f'(s) | > 0
$,
$
  \operatorname{supp}( g ) \subseteq [ \tau_2, \tau_3 ]
$,
$
  \int_{ \R } \left| g(s) \right|^2 ds >0
$,
$
  \operatorname{supp}( h ) \subseteq [ \tau_3, \infty )
$,
and
$
  \int_{ \tau_3 }^{ T } h(s) \, ds \neq 0
$.
\end{ex}

\section{Lower error bounds for general strong approximations}
\label{strong}

In Theorem~\ref{t1} below we provide lower bounds for the error of any 
strong approximation
of $ X^{ \psi }( T ) $ for the processes $ X^{ \psi } $
from Section~\ref{sec:setting} based on the whole path of $ ( W(t) )_{ t \in [0,T] } $
up to a  time interval $ ( t_0 , t_1 ) \subseteq [ 0,  \tau_1/ 2  ] $.
The main tool for the proof of Theorem \ref{t1} is the following
simple symmetrization argument,
which is a special case of the concept of radius of information used in information based complexity, see~\cite{TWW88}.

\begin{lemma}
\label{symm}
Let
$ ( \Omega, \A, \PP ) $ be a probability space,
let
$
  ( \Omega_1, \A_1 )
$
and
$
  ( \Omega_2, \A_2 )
$
be measurable spaces,
and let
$
  V_1 \colon \Omega\to \Omega_1
$
and
$
  V_2, V_2', V_2'' \colon \Omega \to \Omega_2
$
be random variables such that
\begin{equation}
\label{eq:symm_ass}
  \PP_{ (V_1, V_2) } =
  \PP_{ (V_1, V_2') } =
  \PP_{ (V_1, V_2'') }
  \,
  .
\end{equation}
Then
for all measurable mappings
$
  \Phi \colon \Omega_1\times\Omega_2\to \R
$
and
$
  \varphi\colon \Omega_1\to\R
$
we have
\begin{equation}
  \EE\big[
    |\Phi(V_1,V_2)- \varphi(V_1)|
  \big]
  \ge
  \tfrac{ 1 }{ 2 }
  \,
  \EE\big[
    | \Phi( V_1, V_2' ) - \Phi( V_1, V_2'' ) |
  \big]
  .
\end{equation}
\end{lemma}

\begin{proof}
Observe that
\eqref{eq:symm_ass}
ensures that
\begin{equation}
  \EE\big[
    | \Phi(V_1, V_2) - \varphi(V_1) |
  \big]
  =
  \EE\big[
    |
      \Phi( V_1, V_2' ) - \varphi(V_1)
    |
  \big]
  =
  \EE\big[
    | \Phi(V_1,V_2'') - \varphi(V_1) |
  \big]
  .
\end{equation}
This and the triangle inequality imply that
\begin{equation}
\begin{split}
  \EE\big[
    | \Phi(V_1,V_2) - \varphi(V_1) |
  \big]
& \geq
  \tfrac{ 1 }{ 2 } \,
  \EE\bigl[
    | \Phi(V_1,V_2') - \Phi(V_1,V_2'') |
  \bigr]
 ,
\end{split}
\end{equation}
  which finishes the proof.
\end{proof}

In addition, we employ in the proof
of Theorem~\ref{t1} the following lower bound for the first absolute moment
of the sine of a centered normally distributed random variable.

\begin{lemma}
\label{l2}
Let $ ( \Omega , \A, \PP ) $ be a probability space,
let $ \tau \in [1,\infty) $,
and let $ Y \colon \Omega \to \R $ be a
$ \mathcal{N}( 0, \tau^2 ) $-distributed random variable.
Then
\begin{equation}
  \EE\big[
    | \sin(Y) |
  \big]
  \ge
  \frac{ 1 }{ \sqrt{ 8 \pi } } \cdot
  \exp \Bigl( - \frac{ \pi^2 }{ 8 }
  \Bigr)
  .
\end{equation}
\end{lemma}

\begin{proof}
   We have
\begin{equation}
\begin{split}
&
  \EE\big[
    | \sin(Y) |
  \big]
  =
  \frac{ 1 }{ \sqrt{ 2 \pi } }
  \int_\R
  | \sin( \tau z ) |
  \exp\Bigl(
    -
    \frac{ z^2 }{ 2 }
  \Bigr)
  dz
\\
  &
  \ge
  \frac{
    1
  }{
    \sqrt{ 2 \pi }
  }
  \exp\Bigl( - \frac{ \pi^2 }{ 8 } \Bigr)
  \int_0^{ \frac{ \pi }{ 2 } }
  \left| \sin( \tau z ) \right|
  dz
  =
  \frac{
    1
  }{
    \tau \sqrt{ 2 \pi }
  }
  \exp\Bigl( - \frac{ \pi^2 }{ 8 } \Bigr)
  \int_0^{ \frac{ \tau \pi }{ 2 } }
  \left| \sin(z) \right|
  dz
  .
\end{split}
\end{equation}
This and the fact that
\begin{equation}
  \int_0^{
    \frac{ \tau \pi }{ 2 }
  }
  | \sin(x) | \, dx
\ge
  \int_0^{
    \lfloor \tau \rfloor \cdot \frac{ \pi }{2 }
  }
  |\sin(x)| \, dx
=
  \lfloor \tau \rfloor \cdot
  \int_0^{ \frac{ \pi }{ 2 } }
  \sin(x)\, dx
=
  \lfloor \tau \rfloor
\ge
  \frac{ \tau }{ 2 }
\end{equation}
complete the proof.
\end{proof}

We first prove the announced lower error bound for strong approximation 
of $X^\psi(T)$ in the case of the time interval $(t_0,t_1)$ 
being sufficiently small.

\begin{lemma}
\label{lem:strong_lower}
Assume the setting in Section~\ref{sec:setting},
let $ \alpha_1, \alpha_2, \alpha_3, \Delta, \beta \in (0,\infty) $,
and $ \gamma \in \R $ be given by
\begin{equation}
\begin{aligned}
\label{alphas}
  \alpha_1
   =
  \int_0^{ \tau_1 }
  \left| f(s) \right|^2 ds ,
\;\;
  \alpha_2
 =
  \sup_{ s \in [ 0 , \nicefrac{ \tau_1 }{ 2 } ] }
  | f'(s) |^2
  ,
\;\;
  \alpha_3
  =
  \inf_{ s \in [ 0 , \nicefrac{ \tau_1 }{ 2 } ] }
  | f'(s) |^2
  ,\\
  \Delta  =
  \Bigl|
    \min\Bigl\{
      \frac{ \alpha_1 }{ 2  \alpha_2 }
      ,
      \frac{ 1 }{ \alpha_2 }
    \Bigr\}
  \Bigr|^{ 1 / 3 },
  \;\; 
  \beta 
   = \int_{\tau_2}^{\tau_3} \left| g(s) \right|^2 ds,
  \;\;
   \gamma
   =
  \int_{ \tau_3 }^T
  h(s) \, ds,
\end{aligned}  
\end{equation}
let 
$ \psi \in C^{ \infty }( \R, (0,\infty) ) $
be strictly increasing with
$
  \liminf_{ x \to \infty } \psi( x ) = \infty
$
and
$
  \psi\big( \sqrt{ 2 \beta } \big) = 1
$,
let $ t_0, t_1 \in [ 0,  \tau_1 / 2  ] $
satisfy $ 0 < t_1 - t_0 \leq \Delta $,
and let
$
  u \colon
  C\big(
    [0,t_0] \cup [t_1,T] , \R
  \big) \to \R
$
be measurable.
Then  
$
       \tfrac{
          \sqrt{ 12 }
        }{
          ( t_1 - t_0 )^{ 3 / 2 }
          \sqrt{ \alpha_3 }
        }
        \in
        \psi(\R)
$ 
and 
\begin{equation}
  \EE\Big[
    \big|
      X^{ \psi }_4( T ) -
      u\big(
        ( W(s) )_{ s \in [0, t_0] \cup [t_1, T ] }
      \big)
    \big|
  \Big]
\geq
    \frac{ \left| \gamma \right| }{ 8 \pi^{ 3 / 2 } }
    \exp\Bigl(
      - \tfrac{ 2 }{ \beta }
      \Bigl|
        \psi^{ - 1 }\!\Bigl(
        \tfrac{
          \sqrt{ 12 }
        }{
          ( t_1 - t_0 )^{ 3 / 2 }
          \sqrt{ \alpha_3 }
        }
        \Bigr)
      \Bigr|^2
      -
      \tfrac{ \pi^2 }{ 4 }
    \Bigr)
    .
\end{equation}
\end{lemma}

\begin{proof}
   Define stochastic processes
$ \overline{W}, B \colon [ t_0, t_1 ] \times \Omega \to \R $
and
$ \widetilde{W} \colon \big( [ 0, t_0 ] \cup [ t_1, T ] \big) \times \Omega \to \R $
    by
\begin{equation}
  \overline{W}( t ) = \frac{ (t - t_0) }{ ( t_1 - t_0 ) }
  \cdot
  W( t_1 ) + \frac{ ( t_1 - t ) }{ ( t_1 - t_0 ) } \cdot W( t_0 )
  ,
  \qquad
  B( t )
  = W( t ) - \overline{W}( t )
\end{equation}
   for $ t \in [ t_0, t_1 ] $   
and by
$
  \widetilde{W}( t ) = W( t )
$
for 
$ t \in [ 0, t_0 ] \cup [ t_1, T ] $.
Hence, $ B $ is a Brownian bridge on $ [t_0,t_1] $ and 
$ B $ and $ ( \overline W, \widetilde W) $ are independent.

Let
$ Y_1, Y_2 \colon \Omega \to \R $
be random variables such that 
we have $ \PP $-a.s.\ that
\begin{equation}
\begin{split}
  Y_1
&
  =
  \int_0^{t_0} f(s)\, dW(s) + \int_{t_1}^{\tau_1} f(s)\, dW(s)
  +
  f( t_1 ) \, W( t_1 )
  -
  f( t_0 ) \, W( t_0 )
  -
  \int_{t_0}^{t_1} f'(s) \, \overline W(s) \, ds
  ,
\\
  Y_2
&
  = - \int_{ t_0 }^{ t_1 } f'(s) \, B(s) \, ds
\end{split}
\end{equation}
and put
\begin{equation}\label{var} 
 \sigma_i
  =
  \big(
    \EE\big[
      | Y_i |^2
    \big]
  \big)^{ 1 / 2 }
\end{equation}
for $i\in\{1,2\}$.
By the independence of $B$ and $(\overline W,\widetilde W)$ we have 
independence of $Y_1$ and $Y_2$.
Moreover, for all $ i \in \{ 1, 2 \} $
we have
$
  \PP_{ Y_i } = \mathcal{N}( 0,  \sigma_i^2 )
$.
Furthermore,
It\^o's formula proves that 
we have $ \PP $-a.s.\ that
\begin{equation} \label{v2}
  X_2^{ \psi }( \tau_1 ) = Y_1 + Y_2.
\end{equation}
Therefore, we have $ \PP $-a.s.\ that
\begin{equation}
\label{x5}
  X_4^{ \psi }( T )
  =
  \gamma
  \cdot
  \cos\bigl(
    ( Y_1 + Y_2 ) \,
    \psi\big( X_3^{ \psi }( \tau_3 ) \big)
  \bigr)
  .
\end{equation}

First, we provide estimates on the 
variances $ \left| \sigma_1 \right|^2 $ and $ \left| \sigma_2 \right|^2 $.
The fact that $ B $ is a Brownian bridge on $ [t_0,t_1] $
shows that
for all $ s, u \in [ t_0, t_1 ] $
we have
\begin{equation}
\label{eq:Brownian_Bridge}
\EE\big[
    B(s) B(u)
  \big]
  =
   \frac{
    \left( t_1 - \max\{s,u\} \right)\cdot
    \left( \min\{s,u\} - t_0 \right)
  }{
    \left( t_1 - t_0 \right)
  }
  .
\end{equation}

In addition,
the assumption 
$
  \inf_{ s \in [ 0, \nicefrac{ \tau_1 }{ 2 } ] } | f'(s) | > 0
$
implies that for all $ s, u \in [ 0, \tau_1/ 2  ] $
we have
$
  f'(s) \cdot f'(u)
  =
  \left|
    f'(s) \cdot f'(u)
  \right|
$.
The latter fact 
and \eqref{eq:Brownian_Bridge} yield
\begin{equation}
\label{eq:sigma_var_identity}
\begin{split}
   \left| \sigma_2 \right|^2
&
  =
  \EE\!\left[
    \Bigl|
      \int_{ t_0 }^{ t_1 } f'(s) \, B(s) \, ds
    \Bigr|^2
  \right]
  =
  \int_{ t_0 }^{ t_1 }
  \int_{ t_0 }^{ t_1 }
  f'(s) \, f'(u)
  \,
  \EE\big[
    B(s) B(u)
  \big]
  \,
  ds
  \, du
\\
&
  =
  \int_{ t_0 }^{ t_1 }
  \int_{ t_0 }^{ t_1 }
  \left| f'(s) \right|
  \cdot
  \left| f'(u) \right|
  \cdot
  \frac{
    (
      t_1 - \max\{ s, u \}
    )
    \cdot
    (
      \min\{ s, u \} - t_0
    )
  }{
    \left( t_1 - t_0 \right)
  }
  \, ds \, du
  .
\end{split}
\end{equation}
Furthermore, it is easy to see that
\begin{equation}\label{easy}
 \int_{ t_0 }^{ t_1 }
  \int_{ t_0 }^{ t_1 }
  \frac{
    (
      t_1 - \max\{ s, u \}
    )
    \cdot
    (
      \min\{ s, u \} - t_0
    )
  }{
    \left( t_1 - t_0 \right)
  }
  \, ds \, du
  =
 \frac{
    \left( t_1 - t_0 \right)^3
  }{
    12
  }
  . 
\end{equation}
Combining \eqref{eq:sigma_var_identity} and \eqref{easy}
proves that
\begin{equation}
\label{eq:lem_first_inequality}
  0
  <
  \frac{ \alpha_3 \left( t_1 - t_0 \right)^3 }{ 12 }
  \le
   \left| \sigma_2 \right|^2
  \le
  \frac{ \alpha_2 \left( t_1 - t_0 \right)^3 }{ 12 }
  .
\end{equation}
Next 
\eqref{eq:lem_first_inequality} 
and the assumption 
$
  t_1 - t_0 \leq \Delta
$
imply 
\begin{equation}
\label{eq:sigma2_indicator_bound}
  \left| \sigma_2 \right|^2
\le
 \alpha_2 \left| \Delta \right|^3
=
  \min\!\left\{
    \alpha_1/ 2 , 1
  \right\}
  .
\end{equation}
By \eqref{v2}, by the fact that $ Y_1 $ and $ Y_2 $ are independent 
centered normal variables, and 
by \eqref{eq:sigma2_indicator_bound} we get
\begin{equation}
\begin{split}
  \left| \sigma_1 \right|^2
& =
  \EE\big[
    | Y_1 |^2
  \big]
  =
  \EE\big[
    | Y_1 + Y_2 |^2
  \big]
  -
  \EE\big[
    | Y_2 |^2
  \big]
  -
  2
  \,
  \EE\big[
    Y_1 Y_2
  \big]
 \\ &
   =
  \EE\big[
    | X_2^{ \psi }( \tau_1 ) |^2
  \big]
  -
  \left| \sigma_2 \right|^2
  =
  \alpha_1
  - \left| \sigma_2 \right|^2
\ge
   \alpha_1/ 2 
\ge
   \left| \sigma_2 \right|^2
   ,
\end{split}
\end{equation}
   which jointly with
\eqref{eq:sigma2_indicator_bound} 
   yields 
\begin{equation}
\label{eq:lem_second_inequality}
   \left| \sigma_2 \right|^2
  \le
  \min\!\left\{
     \left| \sigma_1 \right|^2 , 1
  \right\} 
  .
\end{equation}

In the next step 
we put up the framework for an application of Lemma \ref{symm}.
Observe that \eqref{x5} and the assumption 
$
  \gamma 
  \neq 0
$
imply 
\begin{equation}\label{v3}
  \EE\Big[
    \bigl| X_4^{ \psi }( T ) - u( \widetilde{W} ) \bigr|
  \Big]
  =
  \left| \gamma \right|
  \cdot
  \EE\Big[
    \bigl|
      \cos\bigl(
        ( Y_1 + Y_2 )
        \,
        \psi\big(
          X_3^{ \psi }( \tau_3 )
        \big)
      \bigr)
      -
      \tfrac{ 1 }{ \gamma }
      \cdot
      u( \widetilde{W} )
    \bigr|
  \Big]
  .
\end{equation}
     Clearly, 
there exist measurable functions
$
  \Phi_i \colon
  C\big(
    [ 0, t_0 ] \cup [ t_1 , 1 ] , \R
  \big) \to \R
$,
$ i \in \{ 1, 2 \} $,
such that 
we have $ \PP $-a.s.\ that
$
  Y_1 =
  \Phi_1(
    \widetilde{W}
 )
$  
and
$
  X_3^{ \psi }( \tau_3 )
  =
  \Phi_2(
    \widetilde{W}
 )
  $.
Moreover, by the independence of $B$ and $(\overline W,\widetilde W)$ we have independence of $Y_2$ and $\widetilde W$. 
Therefore, 
we have
$
  \PP_{ ( \widetilde W, Y_2 ) }
=
  \PP_{ \widetilde W } \otimes \PP_{ Y_2 }
=
  \PP_{ \widetilde W } \otimes \PP_{ - Y_2 }
=
  \PP_{ ( \widetilde W, - Y_2 ) }
$.
We may thus apply Lemma~\ref{symm}
with
$
  \Omega_1 = C( [0, t_0] \cup [t_1 , 1] , \R )
$,
$
  \Omega_2 = \R
$,
$
  V_1 = \widetilde{W}
$,
$
  V_2 = V_2' = Y_2
$,
$
  V_2'' = - Y_2
$,
$
  \varphi =
  \frac{ 1 }{ \gamma } \cdot u
$,
and
$
  \Phi \colon
    C( [0, t_0] \cup [t_1 , T] , \R )
    \times
    \R \to \R
$
given by
$
\Phi(w,y)
    =
    \cos(
      (
        \Phi_1( w ) + y
      )
      \,
      \psi(
        \Phi_2( w )
      )
    )
$
for $ w \in C( [ 0, t_0 ] \cup [ t_1, T ] , \R ) $,
$ y \in \R $
to obtain 
\begin{equation}
\begin{split}
&
  \EE\Big[
    \bigl|
      \cos\bigl(
        ( Y_1 + Y_2 )
        \,
        \psi\big(
          X_3^{ \psi }( \tau_3 )
        \big)
      \bigr)
      -
      \tfrac{ 1 }{ \gamma }
      \cdot
      u(
        \widetilde W
      )
    \bigr|
  \Big]
\\
&
  \quad
  =
  \EE\Bigl[
    \bigl|
      \cos\bigl(
        \big(
          \Phi_1( \widetilde{W} ) + Y_2
        \big)
        \,
        \psi\big(
          \Phi_2( \widetilde{W} )
        \big)
      \bigr)
      -
      \varphi(\widetilde W)
    \bigr|
  \Big]
\\
  &
  \quad
  \geq
  \tfrac{ 1 }{ 2 }
  \cdot
  \EE\Big[
    \bigl|
      \cos\bigl(
        \big(
          \Phi_1( \widetilde{W} ) + Y_2
        \big)
        \,
        \psi\big(
          \Phi_2( \widetilde{W} )
        \big)
      \bigr)
      -
      \cos\bigl(
        \big(
          \Phi_1( \widetilde{W} ) - Y_2
        \big)
        \,
        \psi\big(
          \Phi_2( \widetilde{W} )
        \big)
      \bigr)
    \bigl|
  \Big]
\\
  &
  \quad
  =
  \tfrac{ 1 }{ 2 }
  \cdot
  \EE\Big[
    \bigl|
      \cos\bigl(
        ( Y_1 + Y_2 )
        \,
        \psi\big(
          X_3^{ \psi }( \tau_3 )
        \big)
      \bigr)
      -
      \cos\bigl(
        ( Y_1 - Y_2 )
        \,
        \psi( X_3^{ \psi }( \tau_3 ) )
      \bigr)
    \bigl|
  \Big].
\end{split}
\end{equation}
The latter estimate 
and the fact that
$
  \forall \, x, y \in \R 
  \colon
  \cos( x ) - \cos( y )
  =
  2 \sin( \frac{ y - x }{ 2 } ) \sin( \frac{ y + x }{ 2 } )
$
imply
\begin{equation}
\label{estim1}
\begin{split}
&
  \EE\Big[
    \bigl|
      \cos\bigl(
        ( Y_1 + Y_2 )
        \,
        \psi\big(
          X_3^{ \psi }( \tau_3 )
        \big)
      \bigr)
      -
      \tfrac{ 1 }{ \gamma }
      \cdot
      u(
        \widetilde W
      )
    \bigr|
  \Big]
\\ 
 &
 \qquad\qquad
  \geq
  \EE\Big[
    \bigl|
      \sin\bigl(
        Y_1 \,
        \psi\big(
          X_3^{ \psi }( \tau_3 )
        \big)
      \bigr)
      \cdot
      \sin\bigl(
        Y_2 \,
        \psi\big(
          X_3^{ \psi }( \tau_3 )
        \big)
      \bigr)
    \bigr|
  \Big]
  .
\end{split}
\end{equation}
Since 
$ (\overline W,\widetilde W)$, $ B $, 
and $(W(t)-W(\tau_2))_{t\in[\tau_2,\tau_3]}$ 
are independent we have independence of $Y_1$, $Y_2$, and $  X_3^{ \psi }( \tau_3 )$ as well.
Moreover, we have
$
  \PP_{ X_3^{ \psi }( \tau_3 ) } = \mathcal{N}( 0 , \beta ).
$
The latter two facts 
and \eqref{estim1}
prove
\begin{equation}
\label{eq:cos_estimate_0}
\begin{split}
&
  \EE\Big[
    \bigl|
      \cos\bigl(
        ( Y_1 + Y_2 )
        \,
        \psi\big(
          X_3^{ \psi }( \tau_3 )
        \big)
      \bigr)
      -
      \tfrac{ 1 }{ \gamma }
      \cdot
      u(
        \widetilde W
      )
    \bigr|
  \Big]
\\
  &
\quad  \geq
  \int_\R
    \EE\Big[
      \big|
        \sin\!\big(
          \psi(x) Y_1
        \big)
      \big|
    \Big]
    \cdot
    \EE\Big[
      \big|
        \sin\!\big(
          \psi(x) Y_2
        \big)
      \big|
    \Big]
    \,
  \PP_{ X_3^{ \psi }( \tau_3 ) }( dx )
\\
  & \quad=
  \int_\R
    \EE\Big[
      \big|
        \sin\!\big(
          \psi(x) Y_1
        \big)
      \big|
    \Big]
    \cdot
    \EE\Big[
      \big|
        \sin\!\big(
          \psi(x) Y_2
        \big)
      \big|
    \Big]
    \,
    \tfrac{ 1 }{ \sqrt{ 2 \pi \beta } }
    \,
    \exp\bigl( - \tfrac{ x^2 }{ 2 \beta } \bigr)
     \,dx.
 \end{split}
\end{equation} 
Next we note that \eqref{eq:lem_second_inequality} ensures that 
$ 1 / \sigma_2 \ge 1 $. 
This, 
the assumption that $ \psi $ is continuous,
the assumption that $ \lim_{ x \to \infty } \psi(x) = \infty $,
and the assumption that $ \psi( \sqrt{ 2 \beta } ) = 1 $ 
show
\begin{equation}\label{new3}
  1 / \sigma_2 \in 
  \bigl[\psi(\sqrt{2\beta}),\infty\bigr) \subset \psi(\R).
\end{equation} 
It follows
\begin{equation}
\label{eq:cos_estimate_00}
\begin{split}
&
  \int_\R
    \EE\Big[
      \big|
        \sin\!\big(
          \psi(x) Y_1
        \big)
      \big|
    \Big]
    \cdot
    \EE\Big[
      \big|
        \sin\!\big(
          \psi(x) Y_2
        \big)
      \big|
    \Big]
    \,
    \tfrac{ 1 }{ \sqrt{ 2 \pi \beta } }
    \,
    \exp\bigl( - \tfrac{ x^2 }{ 2 \beta } \bigr)
     \,dx   
\\
  &
\quad\ge
  \int_{
    \psi^{ - 1 }( 1 / \sigma_2 )
  }^{
    2 \psi^{ - 1 }( 1 / \sigma_2 )
  }
    \EE\Big[
      \big|
        \sin\!\big(
          \psi(x) Y_1
        \big)
      \big|
    \Big]
    \cdot
    \EE\Big[
      \big|
        \sin\!\big(
          \psi(x) Y_2
        \big)
      \big|
    \Big]
    \,
    \tfrac{ 1 }{ \sqrt{ 2 \pi \beta } }
    \,
    \exp\bigl( - \tfrac{ x^2 }{ 2 \beta } \bigr)
     \,dx
\\
  &
\quad \ge
  \frac{
    1
  }{
    \sqrt{ 2 \pi \beta }
  }
  \,
    \exp\Bigl(
      -
      \tfrac{ 2 }{ \beta }
      \,
      \big|
        \psi^{ - 1 }( \tfrac{ 1 }{ \sigma_2 } )
      \big|^2
    \Bigr)
  \int_{
    \psi^{ - 1 }( 1 / \sigma_2 )
  }^{
    2 \psi^{ - 1 }( 1 / \sigma_2 )
  }
  \EE\Big[
    \big|
      \sin\!\big(
        \psi(x) Y_1
      \big)
    \big|
  \Big]
  \cdot
  \EE\Big[
    \big|
      \sin\!\big(
        \psi(x) Y_2
      \big)
    \big|
  \Big]
  \, dx .
\end{split}
\end{equation}

We are now in a position to apply Lemma \ref{l2}. 
Observe that \eqref{eq:lem_second_inequality} and
the assumption that $ \psi $ is strictly increasing imply
that for all
$
  x \in [ \psi^{ - 1 }( 1/\sigma_2  ), \infty )
$,
$
  i \in \{ 1, 2 \}
$
we have 
$
  \sigma_i\psi(x) \ge
\sigma_i /\sigma_2  \ge 1
$.
   Employing
Lemma~\ref{l2} 
   we thus conclude 
that
\begin{equation}
\label{eq:sin_estimate_0}
\begin{split}
&
  \int_{
    \psi^{ - 1 }( 1 / \sigma_2 )
  }^{
    2 \psi^{ - 1 }( 1 / \sigma_2 )
  }
  \EE\big[
    | \sin(\psi(x) Y_1) |
  \big]
  \cdot
  \EE\big[
    | \sin(\psi(x) Y_2) |
  \big]
  \, dx
\\ & \qquad \ge
  \int_{
    \psi^{ - 1 }( 1 / \sigma_2 )
  }^{
    2 \psi^{ - 1 }( 1 / \sigma_2 )
  }
  \left[
    \tfrac{ 1 }{ \sqrt{ 8 \pi } }
    \cdot
    \exp\bigl(
      - \tfrac{ \pi^2 }{ 8 }
    \bigr)
  \right]^2
  \,dx
=
  \frac{ 1 }{ 8 \pi }
  \cdot
  \exp\bigl(
    - \tfrac{ \pi^2 }{ 4 }
  \bigr)
  \cdot
  \psi^{ - 1 }\big(
    \tfrac{ 1 }{ \sigma_2 }
  \big)
  .
\end{split}
\end{equation}
Furthermore,
\eqref{eq:lem_first_inequality}, \eqref{new3},
and the assumption that
$ \psi $ is strictly increasing
ensure that
\begin{equation}
\label{eq:psi_estimate_0}
  \psi^{ - 1 }\big(
    \tfrac{ 1 }{ \sigma_2 }
  \big)
\le
  \psi^{ - 1 }\Bigl(
          \tfrac{ \sqrt{12} }{ \sqrt{\alpha_3} }
       \cdot
    \tfrac{ 1 }{
      ( t_1 - t_0 )^{ 3 / 2 }
    }
  \Bigr)
  .
\end{equation}
Combining \eqref{eq:cos_estimate_0}--\eqref{eq:psi_estimate_0}
proves
\begin{equation}
\label{eq:cos_estimate_1}
\begin{split}
&
  \EE\Big[
    \bigl|
      \cos\bigl(
        ( Y_1 + Y_2 )
        \,
        \psi\big(
          X_3^{ \psi }( \tau_3 )
        \big)
      \bigr)
      -
      \tfrac{ 1 }{ \gamma }
      \cdot
      u(
        \widetilde W
      )
    \bigr|
  \Big]
\\
  & \qquad
\ge
  \frac{
    1
  }{
    \sqrt{ 2 \pi \beta }
  }
  \,
    \exp\Bigl(
      - \tfrac{ 2 }{ \beta }
      \Bigl|
        \psi^{ - 1 }\!\Bigl(
        \tfrac{ \sqrt{ 12 } }{ \sqrt{ \alpha_3 } }
        \cdot
        \tfrac{ 1 }{
          ( t_1 - t_0 )^{ 3 / 2 }
        }
        \Bigr)
      \Bigr|^2
    \Bigr)
    \cdot
  \frac{ 1 }{ 8 \pi }
  \cdot
  \exp\bigl(
    - \tfrac{ \pi^2 }{ 4 }
  \bigr)
  \cdot
  \psi^{ - 1 }\big(
    \tfrac{ 1 }{ \sigma_2 }
  \big)
  .
\end{split}
\end{equation}
    Finally, note that \eqref{new3} and 
the assumption that
$ \psi $ is strictly increasing
   imply 
$
  \sqrt{
    2 \beta
  }
\le
  \psi^{ - 1 }\big(
    \tfrac{ 1 }{ \sigma_2 }
  \big)
$.
   Hence, we derive from 
\eqref{eq:cos_estimate_1}
that
\begin{equation}
\begin{split}
&
  \EE\Big[
    \bigl|
      \cos\bigl(
        ( Y_1 + Y_2 )
        \,
        \psi\big(
          X_3^{ \psi }( \tau_3 )
        \big)
      \bigr)
      -
      \tfrac{ 1 }{ \gamma }
      \cdot
      u(
        \widetilde W
      )
    \bigr|
  \Big]
\\
  &\qquad
\ge
    \exp\Bigl(
      - \tfrac{ 2 }{ \beta }
      \left|
        \psi^{ - 1 }\!\left(
        \tfrac{ \sqrt{ 12 } }{ \sqrt{ \alpha_3 } }
        \cdot
        \tfrac{ 1 }{
          ( t_1 - t_0 )^{ 3 / 2 }
        }
        \right)
      \right|^2
    \Bigr)
    \cdot
  \frac{ 1 }{ 8 \pi^{ 3 / 2 } }
  \cdot
  \exp\bigl(
    - \tfrac{ \pi^2 }{ 4 }
  \bigr)
  .
\end{split}
\end{equation}
This and \eqref{v3} complete the proof
of the lemma.
\end{proof}

We are ready to establish our main result.

\begin{theorem}
\label{t1}
Assume the setting in Section~\ref{sec:setting},
let $ \alpha_1, \alpha_2, \alpha_3, \beta, c, C \in (0,\infty) $,
and $ \gamma \in \R $ 
be given by 
\begin{gather}
\label{newconst0}
  \alpha_1
   =
  \int_0^{ \tau_1 }
  \left| f(s) \right|^2 ds ,
\;\;
  \alpha_2
 =
  \sup_{ s \in [ 0 , \nicefrac{ \tau_1 }{ 2 } ] }
  | f'(s) |^2
  ,
\;\;
  \alpha_3
  =
  \inf_{ s \in [ 0 , \nicefrac{ \tau_1 }{ 2 } ] }
  | f'(s) |^2
  ,
\;\;
  \beta 
   = 
   \int_{\tau_2}^{\tau_3} \left| g(s) \right|^2 ds,
\\
\label{newconst}
%   \;\;
   \gamma
   =
  \int_{ \tau_3 }^T
  h(s) \, ds,
\qquad
 c =
  \frac{
    |\gamma|
  }{
    8
    \,
    \pi^{ 3 / 2 }
    \exp( \tfrac{ \pi^2 }{ 4 } )
  }
  ,
\qquad
  C =
  \frac{
    \sqrt{ 12 }\,
    \max\{
      1 ,
      T^{ 3 / 2 }
      \sqrt{\alpha_2}
    \}
  }{
    \sqrt{\alpha_3}
    \min\{
      1 ,
      \sqrt{\tfrac{\alpha_1}{2}}
    \}
  }
  ,
\end{gather}
let $ \psi \in C^{ \infty }( \R, (0,\infty) ) $
be strictly increasing  with
$
  \liminf_{ x \to \infty } \psi( x ) = \infty
$
and
$
  \psi\big( \sqrt{ 2 \beta } \big) = 1
$,
let
$
  0 \le t_0 < t_1 \le \tau_1/ 2
$,
and let 
$
  u \colon
  C\big(
    [0,t_0] \cup [t_1,T] , \R
  \big) \to \R
$
be measurable. Then 
$
 [C/         ( t_1 - t_0 )^{ 3 / 2 },\infty) \subset
  \psi(\R)
$
and  
\begin{equation}
  \EE\Big[
    \big|
      X_4^{ \psi }( T ) -
      u\big(
        ( W(s) )_{ s \in [0, t_0] \cup [t_1, T ] }
      \big)
    \big|
  \Big]
  \geq
  c
  \cdot
  \exp \Bigl(
    -
    \tfrac{ 2 }{ \beta }
    \cdot
    \big|
      \psi^{-1}\big(
        \tfrac{ C }{
          ( t_1 - t_0 )^{ 3 / 2 }
        }
      \big)
    \big|^2
  \Bigr)
  .
\end{equation}
\end{theorem}

\begin{proof}
Let
$  \Delta \in (0,\infty) $
be  given by \eqref{alphas}.

First, assume $ t_1 - t_0 \leq \Delta $. 
By Lemma~\ref{lem:strong_lower} and by the properties of $\psi$ we then have
\begin{equation}\label{sub1}
\Bigl[\tfrac{
          \sqrt{ 12 }
        }{
          ( t_1 - t_0 )^{ 3 / 2 }
          \sqrt{ \alpha_3 }
        }, \infty\Bigr)
        \subset
        \psi(\R)
\end{equation}
        and 
\begin{equation}
\label{eq:lem_appl_diff_t_small}
  \EE\Big[
    \big|
      X_4^{ \psi }( T ) -
      u\big(
        ( W(s) )_{ s \in [0, t_0] \cup [t_1, T ] }
      \big)
    \big|
  \Big]
\geq
   c\cdot
    \exp\Bigl(
      - \tfrac{ 2 }{ \beta }
      \left|
        \psi^{ - 1 }\!\left(
        \tfrac{
          \sqrt{ 12 }
        }{
          ( t_1 - t_0 )^{ 3 / 2 }
          \sqrt{ \alpha_3 }
        }
        \right)
      \right|^2
    \Bigr)
    .
\end{equation}
It remains to observe that
\begin{equation}\label{cc3}
\frac{
   \sqrt{ 12 }
        }{
          ( t_1 - t_0 )^{ 3 / 2 }
          \sqrt{ \alpha_3 }
        }
        \le 
        \frac{C}{ ( t_1 - t_0 )^{ 3 / 2 } },
\end{equation}
and that $ \psi^{-1} $ is  strictly increasing to obtain the desired result in this case.

Next, assume that $t_1 - t_0 >\Delta$.
Then 
Lemma~\ref{lem:strong_lower} together with the properties of $\psi$ yield
\begin{equation}\label{cc4}
\bigl[\tfrac{
          \sqrt{ 12 }
        }{
          \Delta^{ 3 / 2 }
          \sqrt{ \alpha_3 }
        }
        ,\infty\bigr) 
        \subset
        \psi(\R)
      \end{equation}  
        and
\begin{equation}
  \EE\Big[
    \big|
      X_4^{ \psi }( T ) -
      u\big(
        ( W(s) )_{ s \in [0, t_0] \cup [t_1, T ] }
      \big)
    \big|
  \Big]
\geq
   c\cdot 
    \exp\Bigl(
      - \tfrac{ 2 }{ \beta }
      \left|
        \psi^{ - 1 }\!\left(
        \tfrac{
          \sqrt{ 12 }
        }{
          \Delta^{ 3 / 2 }
          \sqrt{ \alpha_3 }
        }
        \right)
      \right|^2
    \Bigr)
    .
\end{equation}
Since
\begin{equation}\label{ddd4}
\frac{
          \sqrt{ 12 }
        }{
          \Delta^{ 3 / 2 }
          \sqrt{ \alpha_3 }
        }
  = \frac{\sqrt{ 12 }\sqrt{\alpha_2}}{ \sqrt{ \alpha_3 }\min\{1,\sqrt{\tfrac{\alpha_1}{2}}\}}
  \le \frac{\sqrt{ 12 }\sqrt{\alpha_2}}{ \sqrt{ \alpha_3 }\min\{1,\sqrt{\tfrac{\alpha_1}{2}}\}}\cdot \frac{T^{3/2}}{(t_1-t_0)^{3/2}}\le  \frac{C}{ ( t_1 - t_0 )^{ 3 / 2 } }
\end{equation}
and since $ \psi^{-1} $ is strictly increasing, we obtain the claimed result in the actual case as well.
\end{proof}

Theorem~\ref{t1} implies uniform lower bounds for the error of
strong approximations of the solution processes
$ X^{ \psi } $
in Section~\ref{sec:setting}
at time $ T $ based on a finite number of function values of the driving
Brownian motion $ W $.
This is, in particular, the subject of the following
corollary.

\begin{cor}
\label{cor1}
Assume the setting in Section~\ref{sec:setting}, 
let $ \alpha_1, \alpha_2, \alpha_3, \beta, c, C \in (0,\infty) $,
and $ \gamma \in \R $ 
be given by 
\begin{gather}
  \alpha_1
   =
  \int_0^{ \tau_1 }
  \left| f(s) \right|^2 ds ,
\;\;
  \alpha_2
 =
  \sup_{ s \in [ 0 , \nicefrac{ \tau_1 }{ 2 } ] }
  | f'(s) |^2
  ,
\;\;
  \alpha_3
  =
  \inf_{ s \in [ 0 , \nicefrac{ \tau_1 }{ 2 } ] }
  | f'(s) |^2
  ,
\;\;
  \beta 
   = 
   \int_{\tau_2}^{\tau_3} \left| g(s) \right|^2 ds,
\\
   \gamma
   =
  \int_{ \tau_3 }^T
  h(s) \, ds,
\qquad
 c =
  \frac{
    |\gamma|
  }{
    8
    \,
    \pi^{ 3 / 2 }
    \exp( \tfrac{ \pi^2 }{ 4 } )
  }
  ,
\qquad
  C =
  \frac{
    \sqrt{ 12 }\,
    \max\{
      1 ,
      T^{ 3 / 2 }
      \sqrt{\alpha_2}
    \}
  }{
    \sqrt{\alpha_3}
    \min\{
      1 ,
      \sqrt{\tfrac{\alpha_1}{2}}
    \}
  }
  ,
\end{gather}
and let $ \psi \in C^{ \infty }( \R, (0,\infty) ) $
be strictly increasing with
$
  \liminf_{ x \to \infty } \psi( x ) = \infty
$
and
$
  \psi\big( \sqrt{ 2 \beta } \big) = 1
$.
Then
for all
$
  n \in \N \cap
 [  2 T / \tau_1  , \infty )
$
and all measurable
$
  u \colon
  C ([ T/n , T ] , \R ) \to \R
$
we have $[ C n^{ 3 / 2 } T^{ - 3 / 2 },\infty)\subset\psi(\R)$ and
\begin{equation}
\label{eq:cor1_eq1}
  \EE\Big[
    \bigl|
      X_4^{ \psi }( T )
      -
      u\big(
        ( W(s) )_{ s \in [ \nicefrac{ T }{ n }, T ] }
      \big)
    \bigr|
  \Big]
  \geq
  c
  \cdot
  \exp\!\left(
    -
    \tfrac{ 2 }{ \beta }
    \cdot
    \left|
      \psi^{ - 1 }\!\left( \tfrac{ C }{ T^{ 3 / 2 } } \cdot n^{ 3 / 2 } \right)
    \right|^2
  \right)
  ,
\end{equation}
for all
$
  n \in \N
$,
$ s_1, \dots, s_n \in [ 0, T ] $
and all measurable
$
  u \colon
  \R^n \to \R
$
we have $[ 8 C n^{ 3 / 2 } ( \tau_1 )^{ - 3 / 2 } , \infty)\subset\psi(\R)$ and
\begin{equation}
\label{eq:cor1_eq2}
  \EE\Big[
    \bigl|
      X^{ \psi }_4(T)
      -
      u\big(
        W( s_1 ),
        \dots ,
        W( s_n )
      \big)
    \bigr|
  \Big]
  \geq
  c
  \cdot
  \exp\Bigl(
    -
    \tfrac{ 2 }{ \beta }
    \cdot
    \bigl|
      \psi^{ - 1 }\!\bigl( \tfrac{ 8 \, C }{ ( \tau_1 )^{ 3 / 2 } } \cdot n^{ 3 / 2 } \bigr)
    \bigr|^2
  \Bigr)
  ,
\end{equation}
and for all
$ n \in \N \cap [  2 T/ \tau_1 , \infty ) $,
$ s_1, \dots, s_n \in [0,T] $
and all measurable
$ u \colon \R^n \times C( [T/n , T ] , \R ) \to \R $
we have $[ 2^{ 3 / 2 }\, C \cdot n^3 / T^{ 3 / 2 } ,\infty)\subset\psi(\R)$ and
\begin{equation}
\label{eq:cor1_eq3}
  \EE\Big[
    \bigl|
      X^{ \psi }_4( T )
      -
      u\big(
        W( s_1 ),
        \dots ,
        W( s_n )
        ,
        ( W(s) )_{ s \in [ \nicefrac{ T }{ n } , T ] }
      \big)
    \bigr|
  \Big]
  \geq
  c
  \cdot
  \exp\Bigl(
    -
    \tfrac{ 2 }{ \beta }
    \cdot
    \bigl|
      \psi^{ - 1 }\bigl( \tfrac{ 2^{ 3 / 2 } \, C }{ T^{ 3 / 2 } } \cdot n^3 \bigr)
    \bigr|^2
  \Bigr)
  .
\end{equation}
\end{cor}

\begin{proof}
Let $ n \in \N $ with
$
  T / n  \leq  \tau_1 /2 
$
and let
$
  u \colon
  C(
    [  T / n  , T ] , \R
  ) \to \R
$
be a measurable mapping.
Then
Theorem~\ref{t1} 
with $t_0=0$ and $t_1=T/n$ 
implies $[C \cdot n^{ 3 / 2 }/ T^{ 3 / 2 },\infty)\subset\psi(\R)$ 
and 
\begin{equation}
\begin{split}
  \EE\Big[
    \big|
      X_4^{ \psi }( T ) -
      u\big(
        ( W(s) )_{ s \in [ \nicefrac{ T }{ n } , T ] }
      \big)
    \big|
  \Big]
& \geq
  c
  \cdot
  \exp\!\left(
    -
    \tfrac{ 2 }{ \beta }
    \cdot
    \big|
      \psi^{-1}\big(
        \tfrac{ C }{
          (T / n )^{ 3 / 2 }
        }
      \big)
    \big|^2
  \right)
  .
\end{split}
\end{equation}
This establishes \eqref{eq:cor1_eq1}.

Next let $ n \in \N $,  
$ s_1, \dots, s_n \in [0,T] $ and 
let 
$ u \colon \R^{ n + 2 } \to \R $
be a measurable mapping.
Then  there exist
$ \hat{s}_0, \hat{s}_1, \dots, \hat{s}_{ n + 1 } \in [0,T] $
such that
$
  0 = \hat{s}_0 \leq \hat{s}_1 \leq \dots \leq \hat{s}_{ n + 1 }
$
and 
$
  \{ \hat{s}_0 , \hat{s}_1, \dots, \hat{s}_{ n + 1 } \}
  \supseteq
  \{ s_1, \dots, s_n,\tau_1/2 \}
$.
In particular, there exists
$ i \in \{ 1, 2, \dots, n+1 \} $
such that
\begin{equation}
\label{eq:delta_hat_s}
\hat s_i \le \tfrac{\tau_1}{2}\quad\text{and}\quad
  \hat{s}_i - \hat{s}_{ i - 1 }
  \geq
  \tfrac{ \tau_1 }{ 2 \left( n + 1 \right) }
  .
\end{equation}
Using Theorem~\ref{t1}
with 
$ t_0=\hat s_{i-1}$ and $t_1 = \hat s_i$ 
and the fact that $ \psi^{ - 1 } $
is increasing we conclude 
that $[ 8C \,n^{ 3 / 2 }/ \tau_1 ^{ 3 / 2 },\infty)\subset 
[ C (2(n+1))^{3/2}/ \tau_1 ^{ 3 / 2 },\infty)
\subset [ C / (\hat{s}_i - \hat{s}_{i-1})^{ 3 / 2 },\infty)\subset\psi(\R) $ and
\begin{equation}
\begin{split}
&
  \EE\Big[
    \bigl|
      X_4^{ \psi }(T)
      -
      u\big(
        W( \hat{s}_0 )
        ,
        W( \hat{s}_1 )
        ,
        \dots
        ,
        W( \hat{s}_n )
        ,
        W( \hat{s}_{ n + 1 } )
      \big)
    \bigr|
  \Big]
  \\
  & \qquad\qquad
  \geq 
  c
  \cdot
  \exp\Bigl(
    -
    \tfrac{ 2 }{ \beta }
    \cdot
    \bigl|
      \psi^{ - 1 }\bigl(
        \tfrac{ C }{
          \left( \hat{s}_i - \hat{s}_{ i - 1 } \right)^{ 3 / 2 }
        }
      \bigr)
    \bigr|^2
  \Bigr)
\geq
c  \cdot
  \exp\Bigl(
    -
    \tfrac{ 2 }{ \beta }
    \cdot
    \bigl|
      \psi^{ - 1 }\bigl(
        \tfrac{
          8 \, C
        }{
           \tau_1 ^{ 3 / 2 }
        }
        \cdot n^{ 3 / 2 }
      \bigr)
    \bigr|^2
  \Bigr)
  .
\end{split}
\end{equation}
This implies \eqref{eq:cor1_eq2}.

The proof of \eqref{eq:cor1_eq3}
is analogous to the proofs of \eqref{eq:cor1_eq1}
and \eqref{eq:cor1_eq2}.
\end{proof}

In Lemma~\ref{lem:speed_psi} below
we characterize a non-polynomial decay
of the lower bounds in \eqref{eq:cor1_eq1}, \eqref{eq:cor1_eq2}, and \eqref{eq:cor1_eq3}
in Corollary~\ref{cor1} in terms of a exponential growth property of the function $\psi$.
To do so, we recall the following elementary fact.

\begin{lemma}
\label{lem:limits}
Let $ \varphi_1 \colon \R \to [0,\infty) $ be non-decreasing,
let $ \varphi_2 \colon \R \to [0,\infty) $ be non-increasing,
and assume that
$
  \liminf_{ \N \ni n \to \infty }
  \left[
   \varphi_1 ( n ) \cdot \varphi_2( n +1)
  \right]
  = \infty
$.
Then
$
  \liminf_{ \R \ni x \to \infty }
  \left[
    \varphi_1( x ) \cdot \varphi_2( x )
  \right]
  = \infty
$.
\end{lemma}

\begin{proof}[Proof]
By the properties of $ \varphi_1 $ and $ \varphi_2 $ we have
for all $ x \in \R $ that
$ \varphi_1(x) \cdot \varphi_2(x) \ge \varphi_1( \lfloor x \rfloor ) \cdot \varphi_2( \lfloor x \rfloor + 1 ) $. 
Hence
\begin{equation}
  \liminf_{ \R \ni x \to \infty }
  \left[
    \varphi_1( x ) \cdot \varphi_2( x )
  \right]
\ge
  \liminf_{ \N \ni n \to \infty }
  \left[
    \varphi_1( n ) \cdot \varphi_2( n + 1 )
  \right]
  = \infty
  ,
\end{equation}
which completes the proof.
\end{proof}

\begin{rem}
We note that in general it is not possible
to replace
in Lemma~\ref{lem:limits}
the assumption 
$
  \displaystyle{\liminf_{ \N \ni n \to \infty }
  \big[
    \varphi_1( n ) \cdot  \varphi_2( n + 1 )
  \big]
  = \infty}
$
by the weaker assumption 
$
   \displaystyle\liminf_{ \N \ni n \to \infty }
  \left[
     \varphi_1( n ) \cdot  \varphi_2( n )
  \right]
  = \infty
$.
Indeed, using suitable mollifiers
one can construct 
$ \varphi_1,\varphi_2\in C^\infty(\R, [0,\infty)) $ such that 
$\varphi_1$ is non-decreasing with
$
  \forall \, n \in \Z \,\,
  \forall \, 
  x \in [ n, n +  1/2 ]
  \colon
  \varphi_1( x ) = 
  \exp\!\big(
    ( n +  1/2 )^2
  \big)
$
and such that $\varphi_2$ is non-increasing with
$
  \forall \, n \in \Z \,\,
  \forall \, 
  x \in [ n -  1 / 2 , n ]
  \colon
  \varphi_2( x ) = 
  \exp( - n^2
  )
$.
Then
\begin{equation}
\begin{aligned}\label{cc45}
  \liminf_{ \N \ni n \to \infty }
  \left[
   \varphi_1( n ) \cdot \varphi_2( n )
  \right]
 &  =
  \liminf_{ \N \ni n \to \infty }
  \exp\!\big(
    (n + 1/2)^2
    - n^2
  \big)
  =
  \infty,\\
  \liminf_{ \N \ni n \to \infty }
  \left[
    \varphi_1( n ) \cdot \varphi_2( n + 1 )
  \right]
 & =
  \liminf_{ \N \ni n \to \infty }
  \exp\!\big(
   ( n + 1/2)^2
    -
    ( n + 1)^2
  \big)
  =
  0,\\
  \liminf_{ \R \ni x \to \infty }
  \left[
    \varphi_1( x ) \cdot \varphi_2( x )
  \right]
 &  \le
  \liminf_{ \N \ni n \to \infty }
  \left[
    \varphi_1( n + 1/2 ) \cdot \varphi_2( n + 1/2 )
  \right]
  =
  0.
  \end{aligned}
  \end{equation}
\end{rem}

\begin{lemma}
\label{lem:speed_psi}
Let $ \eta_1, \eta_2, \eta_3 \in (0,\infty) $
and let $ \psi \colon \R \to (0,\infty) $
be strictly increasing and continuous
with 
$
  \liminf_{ x \to \infty } \psi( x ) = \infty
$.
Then
$
  \forall \, q \in (0,\infty) 
  \colon
  \liminf_{ \N \ni n \to \infty }
$
$
  \big[
    n^q \cdot
    \exp\!\big(
      - \eta_1
      \left| \psi^{ - 1 }( \eta_2 n^{ \eta_3 } ) \right|^2
    \big)
  \big]
  = \infty
$
if and only if
$
  \forall \, q \in (0,\infty) 
  \colon
  \liminf_{ \R \ni x \to \infty }
  \left[
    \psi(x) \cdot \exp( - q x^2 )
  \right]
  = \infty
$.
\end{lemma}

\begin{proof}
We use Lemma~\ref{lem:limits} 
with
$
\varphi_1(x) = x^q
$
and
$
  \varphi_2(x) = \exp(- \eta_1
      \left| \psi^{ - 1 }( \eta_2 x^{ \eta_3 } ) \right|^2
    \big)
$
for $ x \in \R $
to obtain
\begin{equation}
\begin{aligned}
&
\Big(
  \forall \, q \in (0,\infty) \colon
  \liminf_{ \N \ni n \to \infty }
  \big[
    n^q \cdot
    \exp\!\big(
      - \eta_1
      \left| \psi^{ - 1 }( \eta_2 n^{ \eta_3 } ) \right|^2
    \big)
  \big]
  = \infty
\Big)
\\
\Leftrightarrow
\;
&
\Big(
  \forall \, q \in (0,\infty) \colon
  \liminf_{ \R \ni x \to \infty }
  \big[
    x^q \cdot
    \exp\!\big(
      - \eta_1
      \left| \psi^{ - 1 }( \eta_2 x^{ \eta_3 } ) \right|^2
    \big)
  \big]
  = \infty
\Big).
\end{aligned}
\end{equation}
Furthermore, 
\begin{equation}
\begin{split}
&
\Big(
  \forall \, q \in (0,\infty) \colon
  \liminf_{ \R \ni x \to \infty }
  \big[
    x^q \cdot
    \exp\!\big(
      - \eta_1
      \left| \psi^{ - 1 }( \eta_2 x^{ \eta_3 } ) \right|^2
    \big)
  \big]
  = \infty
\Big)
\\
\Leftrightarrow
\;
&
\Big(
  \forall \, q \in (0,\infty) \colon
  \liminf_{ \R \ni x \to \infty }
  \big[
    x^{ \eta_3 q }
    \cdot
    \exp\!\big(
      - \eta_1
      \left| \psi^{ - 1 }( \eta_2 x^{ \eta_3 } ) \right|^2
    \big)
  \big]
  = \infty
\Big)
\\
\Leftrightarrow
\;
&
\Big(
  \forall \, q \in (0,\infty) \colon
  \liminf_{ \R \ni x \to \infty }
  \big[
    x^q
    \cdot
    \exp\!\big(
      - \eta_1
      \left| \psi^{ - 1 }( \eta_2 x ) \right|^2
    \big)
  \big]
  = \infty
\Big)
\\
\Leftrightarrow
\;
&
\Big(
  \forall \, q \in (0,\infty) \colon
  \liminf_{ \R \ni x \to \infty }
  \big[
    x
    \cdot
    \exp\!\big(
      -
      \tfrac{ \eta_1 }{ q }
      \left| \psi^{ - 1 }( \eta_2 x ) \right|^2
    \big)
  \big]
  = \infty
\Big)\\
\Leftrightarrow
\;
&
\Big(
  \forall \, q \in (0,\infty) \colon
  \liminf_{ \R \ni x \to \infty }
  \big[
    x
    \cdot
    \exp\!\big(
      -
      \tfrac{ \eta_1 }{ q }
      \left| \psi^{ - 1 }( x ) \right|^2
    \big)
  \big]
  = \infty
\Big)
.
\end{split}
\end{equation}
Using the properties of $\psi$ we have
\begin{equation}
\begin{split}
&
\Big(
  \forall \, q \in (0,\infty) \colon
  \liminf_{ \R \ni x \to \infty }
  \big[
    x
    \cdot
    \exp\!\big(
      -
      \tfrac{ \eta_1 }{ q }
      \left| \psi^{ - 1 }( x ) \right|^2
    \big)
  \big]
  = \infty
\Big)
\\
\Leftrightarrow
\;
&
\Big(
  \forall \, q \in (0,\infty) \colon
  \liminf_{ \R \ni x \to \infty }
  \big[
    \psi( x )
    \cdot
    \exp\!\big(
      -
      \tfrac{ \eta_1 }{ q }
      x^2
    \big)
  \big]
  = \infty
\Big)
\\
\Leftrightarrow
\;
&
\Big(
  \forall \, q \in (0,\infty) \colon
  \liminf_{ \R \ni x \to \infty }
  \big[
    \psi( x )
    \cdot
    \exp\!\big(
      -
      q x^2
    \big)
  \big]
  = \infty
\Big)
 ,
\end{split}
\end{equation}
which completes the proof.
\end{proof}

As a immediate consequence of \eqref{eq:cor1_eq3} in Corollary~\ref{cor1} and
Lemma~\ref{lem:speed_psi} we get a non-polynomial decay of the error
of any strong approximation of $X^\psi(T)$ based on $ n \in \N $ evaluations of the
driving Brownian motion $W$ and the path of $W$ starting from time $T/n$
if $\psi$ satisfies the exponential growth condition stated in Lemma~\ref{lem:speed_psi}.

\begin{cor}
\label{cor1b}
Assume the setting in Section~\ref{sec:setting},
let $ \beta \in (0,\infty) $ be given by $ \beta = \int_{\tau_2}^{\tau_3} \left| g(s) \right|^2 ds $,
and assume that $ \psi \in C^{ \infty }( \R, (0,\infty) ) $
is strictly increasing with the property that
$
  \psi\big( \sqrt{ 2 \beta } \big) = 1
$
and
$
  \forall \, q \in (0,\infty) \colon
  \liminf_{ x \to \infty }
  \left[
    \psi( x )
    \cdot
    \exp( - q x^2 )
  \right]
  = \infty
$.
Then
for all
$
  q \in (0,\infty)
$
we have 
\begin{equation}
  \liminf_{  n \to \infty }
  \biggl(
  n^q
  \cdot
  \inf_{
    \substack{
      u \colon
      \R^n \times C( [ T / n , T ] , \R ) \to \R
    \\
      \text{measurable, }
      s_1, \dots, s_n \in [0,T]
    }
  }
  \EE\Big[
    \bigl|
      X^{ \psi }_4(T)
      -
      u\big(
        W( s_1 ) , \dots, W( s_n ) ,
        ( W(s) )_{ s \in [ \nicefrac{ T }{ n }, T ] }
      \big)
    \bigr|
  \Big]
  \bigg)
  =
  \infty
  .
\end{equation}
\end{cor}

The following result shows that the smallest possible error 
for strong approximation of $X^\psi(T)$ based on $ n \in \N $ evaluations of the
driving Brownian motion $W$ and the path of $W$ starting from time $T/n$
may decay arbitrarily slow.

\begin{cor}
\label{cor2}
Assume the setting in Section~\ref{sec:setting},
let $ \beta \in (0,\infty) $ be given by 
$
  \beta = \int_{\tau_2}^{\tau_3} \left| g(s) \right|^2 ds,
$
and let $( a_n)_{n\in\N} \subset (0,\infty) $ satisfy
$
  \limsup_{ n \to \infty } a_n = 0
$.
Then there exist a real number $ \kappa \in (0,\infty) $
and a strictly increasing function
$
  \psi \in C^{ \infty }( \R , (0,\infty) )
$
with
$
  \liminf_{ x \to \infty } \psi( x ) = \infty
$
and
$
  \psi\big( \sqrt{ 2 \beta } \big) = 1
$
such that for all $ n \in \N $,
$ s_1, s_2, \dots, s_n \in [0,T] $
and all measurable
$
  u \colon \R^n \times C\big( [ T/n , T ],\R \big) \to \R
$
we have 
\begin{equation}
  \EE\Big[
    \big|
      X^{ \psi }_4( T ) -
      u\big(
        W( s_1 ), \dots, W( s_n ) , ( W(s) )_{ s \in [T/n,T] }
      \big)
    \big|
  \Big]
  \geq
  \kappa\cdot a_n
  .
\end{equation}
\end{cor}

\begin{proof}
Without loss of generality we may assume that the sequence
$
  ( a_n )_{ n \in \N }
$
is strictly decreasing.
Let $c, C\in (0,\infty) $ be
given by \eqref{newconst} and put  
$
  \tilde{C} = 
  2^{ 3 / 2 } C /  T^{ 3 / 2 }
$.
Choose 
$ 
  n_0 \in \N \cap [ \nicefrac{ 2 T }{ \tau_1 } , \infty ) 
$ 
such that
for all
$ n \in \{ n_0, n_0 + 1, \dots \} $
we have
\begin{equation}\label{vr}
  a_n < 1 < \tilde{C} \cdot n^{ 3 }
  \qquad
  \text{and}
  \qquad 
  \tfrac{ \beta }{ 2 }
  \ln\!\big(
    \tfrac{ 1 }{ a_n }
  \big)
  >
  2 \beta,
\end{equation}
and let 
$
  ( b_n )_{ n \in \{ n_0 - 1, n_0 , \dots \} } 
  \subset (0,\infty)
$ 
be such that
$
  b_{ n_0 - 1 } = \sqrt{ 2 \beta }
$ 
and such that for all $ n \in \{ n_0, n_0 + 1, \dots \} $
we have
\begin{equation}
  b_n =
  \Big[
    \tfrac{ \beta }{ 2 }
    \ln\!\big(
      \tfrac{ 1 }{ a_n }
    \big)
  \Big]^{ 1 / 2 }
  .
\end{equation}
Note that 
$ 
  ( b_n )_{ n \in \{ n_0 - 1, n_0 , \dots \} } 
$ 
is strictly increasing and satisfies
$
  \lim_{ n \to \infty } b_n = \infty
$.

Next let
$
  \psi \colon \R \to (0,\infty)
$
be the function with the property that
for all 
$ n \in \{ n_0, n_0 + 1, \dots \} $,
$ x \in \R $
we have
\begin{equation}
  \psi(x) =
  \begin{cases}
    1 -
    \exp\!\big(
      \frac{ 1 }{
        ( x - b_{ n_0 - 1 } )
      }
    \big)
    ,
  &
    \text{if }
    x < b_{ n_0 - 1 }
    ,
\\[1ex]
  1 ,
  &
  \text{if } x = b_{ n_0 - 1 } ,
\\[1ex]
  1 +
  \displaystyle{
    \frac{
       \tilde{C} \cdot n_0^{ 3  } - 1
    }{
        1 +
        \exp\!\big(
          \frac{ 1 }{ ( x - b_{ n_0 - 1 } ) }
          -
          \frac{ 1 }{
            ( b_{n_0} - x )
          }
        \big)
    }
  }
  ,  
  & 
  \text{if } x \in ( b_{ n_0 - 1 } , b_{n_0}),
\\[1ex]
  \tilde{C} \cdot n^{ 3  } ,
  &
  \text{if } x = b_n \text{ and } n\ge n_0,
\\[1ex]
 \tilde{C} \cdot (n-1)^{ 3  }  +
  \displaystyle{
    \frac{
       \tilde{C} \cdot n^{ 3  }  -  \tilde{C} \cdot (n-1)^{ 3 } 
    }{
              1 +
        \exp\!\big(
          \frac{ 1 }{ ( x - b_{ n - 1 } ) }
          -
          \frac{ 1 }{
            ( b_n - x )
          }
        \big)
          }
  }
  ,
  &
  \text{if }
  x \in ( b_{ n - 1 } , b_n ) \text{ and } n > n_0
  .
\end{cases}
\end{equation}
Then 
$ \psi $ is strictly increasing, positive,
and infinitely often differentiable 
and 
$ \psi $ 
satisfies
$
  \psi\big( \sqrt{ 2 \beta } \big) = 1,
$
$
  \liminf_{ x \to \infty } \psi( x ) = \infty
$,
and 
$
  \psi( \R ) = (0,\infty)
$.

In the next step let 
$ \varepsilon_n \in [0,\infty) $,
$ n \in \N $,
be the real numbers with the property that for all 
$ n \in \N $ we have
\begin{equation}
\varepsilon_n =  \inf_{ s_1, \dots, s_n \in [0,T] }\,\,
  \inf_{
    \substack{
      u \colon
      \R^n \times C( [ T / n , T ] , \R ) \to \R
    \\
      \text{measurable}
    }
  }
  \EE\Big[
    \big|
      X_4^{ \psi }(T)
      -
      u\big( W( s_1 ) , \dots, W( s_n ) , ( W(s) )_{ s \in [ \nicefrac{ T }{ n } , T ] } )
    \big|
  \Big].
\end{equation}
Estimate~\eqref{eq:cor1_eq3} in Corollary \ref{cor1} yields
that for all $ n \in \{ n_0, n_0 + 1, \dots \} $ we have
\begin{equation}
\label{eq:applicaton_cor1}
  \varepsilon_n  \geq
c
  \cdot
  \exp\!\left(
    - \tfrac{ 2 }{ \beta }
    \cdot
    \big|
      \psi^{ - 1 }\big(
        \tilde{C} \cdot n^{ 3 }
      \big)
    \big|^2
  \right) 
  = c  \cdot\exp\!\left( - \tfrac{ 2 }{ \beta } \cdot \left| b_n \right|^2 \right)=  c  \cdot a_n
  .
\end{equation}
Since the sequence $(\varepsilon_n)_{n\in\N}$ is non-increasing, we have 
for every $ n \in \{ 1, 2, \dots, n_0 \} $ that
$ \varepsilon_n \ge \varepsilon_{n_0} \ge c \cdot a_{ n_0 } $.
We therefore conclude that for all $n\in\N$ we have
\begin{equation}\label{ds3}
\varepsilon_n \ge c  \cdot \min\{1,a_{n_0}/a_n\}\cdot  a_n \ge \frac{c \,a_{n_0}}{a_1}\cdot a_n, 
\end{equation}
which completes the proof of the corollary with $\kappa= c\cdot a_{n_0}/a_1$.
\end{proof}

Next we extend the result in Corollary \ref{cor2} to approximations that may use
finitely many evaluations of the Brownian path as well as the whole Brownian path 
starting from some arbitrarily small positive time.

\begin{cor}
\label{cor4}
Assume the setting in Section~\ref{sec:setting},
let $ \beta \in (0,\infty) $ be given by 
$ 
  \beta = \int_{\tau_2}^{\tau_3} \left| g(s) \right|^2 ds
$, 
and let $( a_n)_{n\in\N} \subset (0,\infty) $
and 
$ ( \delta_n)_{n\in\N} \subset (0,\infty) $ satisfy
$
\lim_{ n \to \infty } a_n =  \lim_{ n \to \infty } \delta_n = 0
$.
Then there exist a real number $ \kappa \in (0,\infty) $
and a strictly increasing function
$
  \psi \in C^{ \infty }( \R , (0,\infty) )
$
with
$
  \liminf_{ x \to \infty } \psi( x ) = \infty
$
and
$
  \psi\big( \sqrt{ 2 \beta } \big) = 1
$
such that for all $ n \in \N $,
$ s_1, s_2, \dots, s_n \in [0,T] $
and all measurable
$
  u \colon \R^n \times C\big( [ \delta_n , T ],\R \big) \to \R
$
we have 
\begin{equation}
  \EE\Big[
    \big|
      X^{ \psi }_4( T ) -
      u\big(
        W( s_1 ), \dots, W( s_n ) , ( W(s) )_{ s \in [\delta_n,T] }
      \big)
    \big|
  \Big]
  \geq
  \kappa\cdot a_n
  .
\end{equation}
\end{cor}

\begin{proof}
Without loss of generality we may assume that the sequence
$
  (\delta_n )_{ n \in \N }
$
is strictly decreasing. 
Let 
$
  ( k_n )_{n\in\N}
  \subset (0,\infty)
$ 
be the strictly increasing sequence of
positive integers 
with the property that for all $ n \in \N $
we have
\begin{equation}\label{nr1}
  k_n = 
  \lceil \nicefrac{ T }{ \delta_n } \rceil + n.
\end{equation}
Moreover, 
let
$
  (\tilde a_n)_{n\in\N}\subset (0,\infty)
$
be a sequence such that 
for all $ n \in \N $ we have
\begin{equation}\label{nr2}
  \tilde a_{k_n} = a_n
\end{equation}
and 
$
  \lim_{ m \to \infty } \tilde a_m = 0
$.
Then Corollary \ref{cor2} implies that there exist 
a real number $ \kappa \in (0,\infty) $
and a strictly increasing function
$
  \psi \in C^{ \infty }( \R , (0,\infty) )
$
with
$
  \liminf_{ x \to \infty } \psi( x ) = \infty
$
and
$
  \psi\big( \sqrt{ 2 \beta } \big) = 1
$
such that for all $ n \in \N $,
$ s_1, s_2, \dots, s_n \in [0,T] $
and all measurable
$
  \tilde u \colon \R^n \times C\big( [ \nicefrac{ T }{ n } , T ],\R \big) \to \R
$
we have 
\begin{equation}\label{nr3}
  \EE\Big[
    \big|
      X^{ \psi }_4( T ) -
      \tilde u\big(
        W( s_1 ), \dots, W( s_n ) , ( W(s) )_{ s \in [T/n,T] }
      \big)
    \big|
  \Big]
  \geq
  \kappa\cdot \tilde a_n
  .
\end{equation}

Let $ n \in \N $, let
$
  u \colon \R^n \times C\big( [ \delta_n , T ],\R \big) \to \R
$
be a measurable mapping,
and let $ s_1, s_2, \dots, s_n \in [0,T] $. Note that \eqref{nr1} implies 
$ \delta_n \ge \nicefrac{ T }{ k_n } $ and $ k_n \ge n $. 
Put $ s_m = s_n $ for $ m \in \{ n+1, n+2, \dots, k_n \} $. Clearly, there exists a measurable mapping
$
  \tilde u \colon \R^{k_n} \times C\big( [ T/k_n , T ],\R \big) \to \R
$
such that 
$ 
  u\big(
        W( s_1 ), \dots, W( s_n ) , ( W(s) )_{ s \in [\delta_n,T] }
      \big) 
  = \tilde u\big(
        W( s_1 ), \dots, W( s_{k_n} ) , ( W(s) )_{ s \in [T/k_n,T] } \big) 
$. 
Hence, by \eqref{nr3} and by \eqref{nr2}, we have
\begin{equation}
  \EE\Big[
    \big|
      X^{ \psi }_4( T ) -
      u\big(
        W( s_1 ), \dots, W( s_n ) , ( W(s) )_{ s \in [\delta_n,T] }
      \big)
    \big|
  \Big]
  \geq
  \kappa\cdot \tilde a_{k_n} = \kappa\cdot a_n,
\end{equation}
which completes the proof.
\end{proof}

\section{Upper error bounds for 
    the Euler-Maruyama scheme
}
\label{sec:upper_bounds}

A classical 
method for strong approximation of SDEs 
is provided by the Euler-Maruyama scheme.
In Theorem~\ref{t2} below we establish upper bounds for the
root mean square errors of Euler-Maruyama approximations of
 $ X^{ \psi }(T) $ for 
the processes $ X^{ \psi } $, $ \psi \in C^{ \infty }( \R, (0,\infty) ) $, from Section~\ref{sec:setting}.
In particular, it turns out that in the case of non-polynomial convergence
the Euler-Maruyama approximation may still
perform asymptotically optimal, at least on a logarithmic scale, see
Example \ref{ex34} below for details.

We first provide some elementary bounds for tail probabilities
of normally distributed random variables.

\begin{lemma}
\label{lem:gaussian}
Let $ ( \Omega, \mathcal{A}, \PP ) $ be a probability space,
let $ x \in \R $,
and let $ Z \colon \Omega \to \R $ be a standard normal random variable.
Then 
\begin{equation}
\label{eq:PZ_0}
  \PP\big(
    Z \geq x
  \big)
  \leq
  \tfrac{ 1 }{ \sqrt{ 2 } }
  \cdot
  \exp \bigl(-\tfrac{ x | x | }{ 2 } \bigr)
  .
\end{equation}
\end{lemma}

\begin{proof}
For every $y\in [0,\infty)$ we have
\begin{equation}
\begin{aligned}
(y+x)^2 - x|x| - \tfrac{y^2}{2} & = \tfrac{1}{2}(y^2 + 4xy + 2x(x-|x|)) =
 \tfrac{1}{2}(y^2 + 4xy + 4x^2\1_{(-\infty,0]}(x)) \ge 0. 
\end{aligned}
\end{equation}
Hence
\begin{equation}
\begin{aligned}
\PP(Z\ge x) & = \int_0^\infty   \tfrac{ 1 }{ \sqrt{ 2 \pi } }
  \cdot
  \exp\bigl(-\tfrac{ (y+x)^2 }{ 2 }\bigr)\, dy \\
   & \le  
  \exp \bigl(-\tfrac{ x | x | }{ 2 } \bigr) \int_0^\infty   \tfrac{ 1 }{ \sqrt{ 2 \pi } }
  \cdot\exp\bigl(-\tfrac{ y^2 }{ 4 }\bigr)\, dy =  \tfrac{ 1 }{ \sqrt{ 2 } }
  \cdot
  \exp \bigl(-\tfrac{ x | x | }{ 2 } \bigr),
\end{aligned}
\end{equation}
which completes the proof.
\end{proof}

\begin{lemma}
\label{lem:gaussian2}
Let $ ( \Omega, \mathcal{A}, \PP ) $ be a probability space,
let $ \sigma \in [0,\infty) $, $ c \in (0,\infty) \cap [ \sigma, \infty ) $,
and let $ Z \colon \Omega \to \R $ be a $ \mathcal{N}( 0, \sigma^2 ) $-distributed random variable.
Then for all $ x \in \R $ we have
\begin{equation}
\label{eq:PZ_0_b}
  \PP\big(
    Z \geq x
  \big)
  \leq
  \exp\!\big(
    - \tfrac{ [ \max\{ x, 0 \} ]^2 }{ 2 c^2 }
  \big)
  .
\end{equation}
\end{lemma}

\begin{proof}
In the case $ \sigma = 0 $ we note that for all $ x \in \R $ we have
\begin{equation}
  \PP\big(
    Z \geq x
  \big)
=
  \mathbbm{1}_{ ( - \infty , 0 ] }( x )
\leq
  \exp\big(
    - \tfrac{ [ \max\{ x, 0 \} ]^2 }{ 2 c^2 }
  \big)
  .
\end{equation}
In the case $ \sigma > 0 $ we use 
Lemma~\ref{lem:gaussian} to obtain 
that for all $ x \in [0,\infty) $ we have
\begin{equation}
  \PP\big(
    Z \geq x
  \big)
  =
  \PP\big(
    \tfrac{ Z }{ \sigma } \geq \tfrac{ x }{ \sigma }
  \big)
  \leq
  \tfrac{ 1 }{ \sqrt{ 2 } }
  \cdot
  \exp\!\big(
    - \tfrac{ x^2}{ 2 \sigma^2 }
  \big)
   \leq
  \exp\!\big(
    - \tfrac{ x^2  }{ 2 c^2 }
  \big)
  ,
\end{equation}
which completes the proof.
\end{proof}

Next we relate exponential growth of a continuously differentiable function to exponential growth of its derivative.

\begin{lemma}
\label{lem:existence}
Let
   $\psi\in C^1(\R,\R)$
satisfy 
$
  \forall \, q \in (0,\infty) 
  \colon
  \liminf_{ x \mapsto \infty }
  \left[
    \psi( x )
    \cdot
    \exp( - q x^2 )
  \right]
  = \infty
$
and assume that
$
  \psi'
$
is non-decreasing.
Then  
$
\forall \, q \in \R 
  \colon
  \liminf_{ x \mapsto \infty }
  \left[
    \psi'( x )
    \cdot
    \exp( - q x^2 )
  \right]
  = \infty
$.
\end{lemma}

\begin{proof}
Since 
$
  \forall \, q \in (0,\infty)
  \colon
  \liminf_{  x \to \infty }
  \big[
    \psi( x ) \cdot
    \exp( - q x^2 )
  \big]
  = \infty
$,
we have
\begin{equation}
\label{eq:psi_limit}
  \forall \, q \in \R
  \colon
  \liminf_{  x \to \infty }
  \left[
    \psi( x ) \cdot
    \exp( - q x^2 )
  \right]
  = \infty
  .
\end{equation}
By 
the fundamental theorem of calculus and
the assumption that $ \psi' $ is increasing we obtain 
for all $ x \in (0,\infty) $  that
\begin{equation}
  \psi'(x)
=
  \frac{ 1 }{ x }
  \int_0^x \psi'(x) \, dy
\ge
  \frac{ 1 }{ x }
  \int_0^x \psi'(y) \, dy
=
  \frac{ \psi(x) - \psi(0) }{ x }
  .
\end{equation}
Hence, 
for all $ q \in \R $ we have
\begin{equation}
\begin{split}
&
  \liminf_{  x \to \infty }
  \big[
    \psi'(x) \cdot \exp( - q x^2 )
  \big]
\ge
  \liminf_{  x \to \infty }
  \left[
    \frac{ \psi(x) - \psi(0) }{ x \cdot \exp( q x^2 ) }
  \right]
\ge
  \liminf_{  x \to \infty }
  \left[
    \frac{ \psi(x) - \frac{ 1 }{ 2 } \psi(x) }{ x \cdot \exp( q x^2 ) }
  \right]
\\ &
\qquad \qquad =
  \liminf_{  x \to \infty }
  \left[
    \frac{ \psi(x) }{ 2 x \cdot \exp( q x^2 ) }
  \right]
\ge
  \liminf_{  x \to \infty }
  \left[
    \psi(x) \cdot \exp( - 2 q x^2 )
  \right]
=
  \infty
  ,
\end{split}
\end{equation}
which completes the proof.
\end{proof}

We turn to the analysis of the Euler-Maruyama scheme for strong approximation of 
SDEs in the setting of 
Section~\ref{sec:setting}.

\begin{theorem}
\label{t2}
Assume the setting in Section~\ref{sec:setting},
assume that $ \tau_1 < \tau_2 $, 
let $ \beta \in (0,\infty) $ be given by
$ 
  \beta = \int_{\tau_2}^{\tau_3} \left| g(s) \right|^2 ds
$,
let $ \delta \in (0,1) $, 
let $ \psi \in C^\infty(\R,(0,\infty)) $ be strictly increasing 
such that 
$
  \psi\big( \sqrt{ 2 \beta } \big) = 1
$,
such that 
$
 \forall \, q \in (0,\infty) \colon
 \liminf_{ x \mapsto \infty }
  \left[
    \psi( x )
    \cdot
    \exp( - q x^2 )
  \right]
  = \infty
$,
and such that
$ \psi' $ is strictly inreasing,
and let 
$ 
  \widehat{X}^{ ( \psi, n ) } 
  \colon \{ 0, 1, \dots, n \} \times \Omega \to \R^4 
$,
$ n \in \N $,
satisfy 
for all 
$ n \in \N $, $ k \in \{ 0, 1, \dots, n - 1 \} $
that
$\widehat X^{(\psi,n)}_0 = 0$ and 
\begin{equation}\label{Euler}
\widehat{X}^{ (\psi,n) }_{ k + 1 } =
  \widehat{X}^{ (\psi,n) }_k
  +
  \mu^{ \psi }( \widehat{X}^{ (\psi,n) }_k ) \, \tfrac{ T }{ n }
  +
  \sigma( \widehat{X}^{ (\psi,n) }_k ) \,
  \big(
    W( \tfrac{ ( k + 1 ) T }{ n } )
    -
    W( \tfrac{ k T }{ n } )
  \big).
\end{equation}
Then there exist real numbers 
$ c \in (0,\infty) $ and
$ n_0 \in \N $ such that 
$ 
  \big[ | n_0 |^\delta , \infty \big) 
  \subset \psi'(\R) 
$ 
and 
such that 
for every $ n \in \{ n_0, n_0 + 1, \dots \} $ 
we have 
\begin{equation}
\label{eq:t2}
  \Big(
    \EE\Big[
      \big\|
        X^{ \psi }( T ) -
        \widehat{X}_n^{ (\psi,n) }
      \big\|_{ \R^4 }^2
    \Big]
  \Big)^{ 1 / 2 }
\leq
  c
  \,
  \bigg[
    \exp\!\left(
      - \tfrac{ 1 }{ c } \cdot \left| \psi^{ - 1 }( n^{ \delta } ) \right|^2
    \right)
    +
    \exp\!\left(
      - \tfrac{ 1 }{ c } \cdot
      \left| ( \psi' )^{ - 1 }( n^{ \delta } ) \right|^2
    \right)
  \bigg]
  .
\end{equation}
\end{theorem}

\begin{proof}
Throughout this proof let
$ \Delta W_j^n \colon \Omega \to \R $,
$ j \in \{ 1, 2, \dots, n \} $,
$ n \in \N $,
be the mappings with the property that for all
$ n \in \N $, $ j \in \{ 1, 2, \dots, n \} $
we have
$ \Delta W_j^n = W( \frac{ j T }{ n } ) - W( \frac{ (j-1) T }{ n } ) $,
let 
$ \beta_n \in \R $, $ n \in \N $,
and
$ \gamma_n \in \R $, $ n \in \N $,
be the real numbers with the property that
for all $ n \in \N $ we have
\begin{equation}
  \gamma_n = \sum_{ j = 1 }^n
  \tfrac{ T }{ n }
  \cdot
  h\bigl( \tfrac{ (j-1) T }{ n } \bigr),
\qquad  
  \beta_n = \sum_{ j = 1 }^n
  \tfrac{ T }{ n }
  \cdot
  \bigl|
    g\bigl( \tfrac{ (j-1) T }{ n } \bigr)
  \bigr|^2
  ,
\end{equation}
and let 
$ \widehat{X}^{ ( \psi, n ) }_{ l, ( \cdot ) } 
\colon \{ 0, 1, \dots, n \} \times \Omega \to \R $,
$ l \in \{ 1, 2, 3, 4 \} $,
$ n \in \N $,
be the stochastic processes with the property that
for all $ n \in \N $,
$ k \in \{ 0, 1, \dots, n \} $
we have
$
  \widehat X^{(\psi,n)}_k = (\widehat X^{(\psi,n)}_{1,k},\dots,\widehat X^{(\psi,n)}_{4,k})
$.
By the properties of $f,g,h$ stated in Section \ref{sec:setting} 
and by the definition of $\mu^\psi$ and $\sigma$ (see \eqref{coeff}), we have for
all $ n \in \N $, $ k \in \{ 0, 1, \dots, n \} $ that
$
  \widehat X^{(\psi,n)}_{1,k} = \tfrac{ k\cdot T }{ n }
$
and
\begin{equation}
\begin{aligned}
\widehat X^{(\psi,n)}_{2,k} & = \sum_{j=1}^k f\bigl( \tfrac{ (j-1) T }{ n } \bigr)
  \cdot
  \Delta W_j
   =  
   \sum_{j=1}^{\min\{k,\lceil n\tau_1/T\rceil\}} f\bigl( \tfrac{ (j-1) T }{ n } \bigr)
  \cdot
  \Delta W_j
  ,
\\
  \widehat{X}^{ (\psi,n) }_{ 3, k }
  &
  = \sum_{ j = 1 }^k
  g\bigl( \tfrac{ (j-1) T }{ n } \bigr)
  \cdot
  \Delta W_j = \sum_{j=\lfloor n\tau_2/T\rfloor+2}^{\min\{k,\lceil n\tau_3/T\rceil\}} g\bigl( \tfrac{ (j-1) T }{ n } \bigr)
  \cdot
  \Delta W_j
  ,
\\
  \widehat{X}^{ (\psi,n) }_{ 4, k }
  &
  =
  \sum_{ j = 1 }^k
  \tfrac{ T }{ n }
  \cdot
  h\bigl( \tfrac{ (j-1) T }{ n } \bigr)
  \cdot
  \cos\!\big(
    \widehat{X}^{ (\psi,n) }_{ 2 , j - 1 }
    \cdot
    \psi(
      \widehat{X}^{ (\psi,n) }_{ 3 , j - 1 }
    )
  \big)\\
&  =
  \sum_{ j = \lfloor n\tau_3/T\rfloor+2 }^k
  \tfrac{ T }{ n }
  \cdot
  h\bigl( \tfrac{ (j-1) T }{ n } \bigr)
  \cdot
  \cos\!\big(
    \widehat{X}^{ (\psi,n) }_{ 2 , j - 1 }
    \cdot
    \psi(
      \widehat{X}^{ (\psi,n) }_{ 3 , j - 1 }
    )
  \big)
  .
\end{aligned}
\end{equation}
In particular, for all $ n \in \N $, $ k \in [ \frac{ n \tau_1 }{ T } , \infty ) \cap \{ 1, 2, \dots, n \} $ we have 
$ \widehat X^{(\psi,n)}_{2,k} = \widehat X^{(\psi,n)}_{2,n} $ 
and 
for all $ n \in \N $, $ k \in [ \frac{ n \tau_3 }{ T } , \infty ) \cap \{ 1, 2, \dots, n \} $ we have
$\widehat X^{(\psi,n)}_{3,k} = \widehat X^{(\psi,n)}_{3,n}$. 
Therefore, 
for all $ n \in \N $
we have
\begin{equation}\label{k4}
  \widehat{X}^{ (\psi,n) }_{ 4, n } =  \sum_{ j = \lfloor n\tau_3/T\rfloor+2 }^k
  \tfrac{ T }{ n }
  \cdot
  h\bigl( \tfrac{ (j-1) T }{ n } \bigr)
  \cdot
  \cos\!\big(
    \widehat{X}^{ (\psi,n) }_{ 2 , n }
    \cdot
    \psi(
      \widehat{X}^{ (\psi,n) }_{ 3 ,n }
    )
  \big)
  = \gamma_n
  \cdot
  \cos\!\big(
    \widehat{X}^{ (\psi,n) }_{ 2 , n }
    \cdot
    \psi(
      \widehat{X}^{ (\psi,n) }_{ 3 ,n }
    )
  \big). 
\end{equation}

We separately analyze the componentwise  mean square errors 
\begin{equation}
  \varepsilon_{ i, n } = \EE\bigl[ |X^\psi_i(T)-\widehat X^{(\psi,n)}_{i,n} |^2
  \bigr]
\end{equation}
for $i\in\{1,\dots,4\}$, $ n \in \N $. 
Clearly, for all $ n \in \N $ we have $\varepsilon_{ 1, n } = 0 $. 
Moreover, It\^{o}'s isometry shows that for all $ n \in \N $ we have
\begin{equation}
\label{eq:second_component}
\begin{aligned}
  \varepsilon_{ 2, n } 
&
  =
  \EE\!\left[
    \left|
    \sum_{ j = 1 }^{ n }
    \int_{ ( j - 1 ) T / n }^{ j T / n }
    \big(
      f(s) - f( \tfrac{ (j-1) T }{ n } )
    \big)
    \, dW(s)
    \right|^2
  \right]
  =
  \sum_{ j = 1 }^n
  \int_{
    ( j - 1 ) T / n
  }^{
    j T / n
  }
  \big|
    f(s) - f( \tfrac{ (j-1) T }{ n } )
  \big|^2
  \,
  ds\\
  &
\leq
    \sup_{ t \in [0, \tau_1 ] }
    | f'(t) |^2
 \cdot
  \sum_{ j = 1 }^n
  \int_{
    ( j - 1 ) T / n
  }^{
    j T / n
  }
  \big|
    s - \tfrac{ (j-1) T }{ n }
  \big|^2
  \,
  ds
=
  \frac{ T^3 }{ 3 n^2 }
\cdot
    \sup_{ t \in [0, \tau_1 ] }
    |f'(t)|^2,
\end{aligned}
\end{equation}
and, similarly,
\begin{equation}
\label{eq:third_component}
  \varepsilon_{ 3, n } 
  \le
\frac{ T^3 }{ 3 n^2 }
\cdot
    \sup_{ t \in [\tau_2, \tau_3 ] }
    |g'(t)|^2,
  \qquad 
\EE\bigl[|\widehat X^{(\psi,n)}_{2,n}|^2\bigr] 
  \le  T\cdot  \sup_{ t \in [0, \tau_1 ] }
    |f(t)|^2
    .
\end{equation}

We turn to the analysis of $\varepsilon_{ 4, n } $, $ n \in \N $. 
For this let $ \gamma \in \R $ be given by
$
  \gamma = \int_{ \tau_3 }^T h(s) \, ds
$
(see \eqref{alphas}). 
From \eqref{k4} we obtain 
\begin{equation}
\label{cv1}
  \varepsilon_{ 4, n }
 \leq
  2 \left| \gamma \right|^2
    \cdot 
    \EE\!\left[
    \big|
      \cos\!\big(
        X^{ \psi }_2( T )
        \cdot
        \psi\big(
          X^{ \psi }_3( T )
        \big)
      \big)
      -
      \cos\!\big(
        \widehat{X}^{ (\psi,n) }_{ 2 , n }
        \cdot
        \psi(
          \widehat{X}^{ (\psi,n) }_{ 3 , n }
        )
      \big)
    \big|^2
  \right]
  +
    2 
    \left|
      \gamma
      -
      \gamma_n
    \right|^2
    .
\end{equation}
Clearly, for all $ n \in \N $ we have
\begin{equation}
\label{eq:gamma_estimate}
\begin{aligned}
  \left|
    \gamma - \gamma_n
  \right|
&
  = \biggl| \sum_{j=1}^n\int_{
      ( j - 1 ) T / n
    }^{
      j T / n} \bigl(h(s)-h\big(
        \tfrac{ (j-1) T }{ n }
      \big)\bigr)\, ds\biggr|
     \\   &
\leq
    \sup_{ t \in [\tau_3, T ] }
    | h'(t) |
 \cdot
  \sum_{ j = 1 }^n
  \int_{
    ( j - 1 ) T / n
  }^{
    j T / n
  }
  \big|
    s - \tfrac{ (j-1) T }{ n }
  \big|
  \,
  ds
=
  \frac{ T^2 }{ 2 n }
\cdot
    \sup_{ t \in [\tau_3,T ] }
    |h'(t)|.
    \end{aligned}
\end{equation}
Using a trigonometric identity, the fact that 
$
  \forall \, x \in \R \colon
   | \sin(x) |\le \min\{ 1, |x| \}
$, 
inequality~\eqref{eq:second_component}, 
the fact that
$
  \PP_{X_3^\psi(T)} = \mathcal N(0,\beta)
$, 
a standard estimate of Gaussian tail probabilities, see, e.g., \cite[Lemma~22.2]{Klenke2008}, 
and the fact that $\psi^{-1}(n^\delta)\ge \psi^{-1}(1) = \sqrt{2\beta}$
we get
for all $ n \in \N $ that
\begin{equation}\label{nummer}
\begin{aligned}
&
  \EE\Big[
    \big|
      \cos\!\big( X^{ \psi }_2( T ) \cdot \psi\big( X^{ \psi }_3( T ) \big)
      \big)
      -
      \cos\!\big(
        \widehat{X}^{ (\psi,n) }_{ 2, n } \cdot \psi\big( X^{ \psi }_3( T ) \big)
      \big)
    \big|^2
  \Big]
\\
&
  \quad
  =
  4
  \cdot
  \EE\bigg[
    \Big|
      \sin\!\Big(
        \tfrac{ 1 }{ 2 }
        \,
        \big(
          X^{ \psi }_2( T )
          -
          \widehat{X}^{ (\psi,n) }_{ 2, n }
        \big)
        \,
        \psi\big(
          X^{ \psi }_3( T )
        \big)
      \Big)
      \sin\!\Big(
        \tfrac{ 1 }{ 2 }
        \,
        \big(
          X^{ \psi }_2( T )
          +
          \widehat{X}^{ (\psi,n) }_{ 2, n }
        \big)
        \,
        \psi\big(
          X^{ \psi }_3( T )
        \big)
      \Big)
    \Big|^2
  \bigg]
\\
&
  \quad
  \le
  4 \cdot
  \EE\bigg[
    \Big|
      \sin\!\Big(
        \tfrac{ 1 }{ 2 }
        \,
        \big(
          X^{ \psi }_2( T )
          -
          \widehat{X}^{ (\psi,n) }_{ 2, n }
        \big)
        \,
        \psi\big(
          X^{ \psi }_3( T )
        \big)
      \Big)
    \Big|^2
  \bigg]
\\
&
  \quad
  \le
  \EE\bigg[
        \big|
          X^{ \psi }_2( T )
          -
          \widehat{X}^{ (\psi,n) }_{ 2, n }
        \big|^2
        \,
      \big|
        \psi\big(
          X^{ \psi }_3( T )
        \big)
      \big|^2
      \,
      \mathbbm{1}_{
        \{
          X^{ \psi }_3( T )
          \leq \psi^{ - 1 }( n^{ \delta } )
        \}
      }
  \bigg]
  + 4 \cdot
  \PP\Big(
      X^{ \psi }_3( T )
      > \psi^{ - 1 }( n^{ \delta } )
  \Big)
\\
&
  \quad
  \le
  n^{ 2 \delta }
  \,
  \EE\Big[
        \big|
          X^{ \psi }_2( T )
          -
          \widehat{X}^{ (n) }_{ 2, n }
        \big|^2
  \Big]
  +  \frac{
    \sqrt{\beta} 
  }{
    \psi^{ - 1 }( n^{ \delta } )
    \sqrt{ 2 \pi }
  }
  \cdot 
  \exp\!\left(
    -
    \tfrac{
      1
    }{
      2
      \beta
    }
    \cdot
      \left|
        \psi^{ - 1 }(
          n^{ \delta }
        )
      \right|^2
  \right)
\\ &
  \quad
  \le
  \frac{ T^3 }{
    3 \, n^{ 2 ( 1 - \delta ) }
  }
    \sup_{ t \in [0, \tau_1 ] }
    |f'(t)|^2
  +
  \frac{
    1
  }{
    2 \sqrt{ \pi }
  }
  \cdot
  \exp\!\left(
    -
    \tfrac{
      1
    }{
      2
      \beta
    }
    \cdot
      \left|
        \psi^{ - 1 }(
          n^{ \delta }
        )
      \right|^2
  \right)
  .
\end{aligned}
\end{equation}
By Lemma \ref{lem:existence} we have $\lim_{x\to\infty}\psi'(x)=\infty$. 
Hence, there exists $ n_1 \in \N $ 
such that $ \big[ | n_1 |^\delta, \infty \big) \subset \psi'\big( [0, \infty) \big) $. 
Put
\begin{equation}
  n_0 = \max\Bigl\{n_1,\bigl\lceil\tfrac{T}{\tau_2-\tau_1}\bigr\rceil\Bigr\}
\end{equation}
and let $ n \in \{ n_0 , n_0 + 1 , \dots \} $. Then 
$(W(s))_{s\in[0,\lceil n\tau_1/T\rceil\cdot T/n]}$ and $(W(s)-W(\tau_2))_{s\in[\tau_2,T]}$ are independent, 
which implies independence of the random variables   $\widehat X^{(\psi,n)}_{2,n}$ and $X_3^\psi(T)-\widehat X_{3,n}^{(\psi,n)}$. 
Using the latter fact as well as the fact that $ \psi' $ is strictly increasing 
and the estimates in \eqref{eq:third_component}, we may proceed analoguously 
to the derivation of \eqref{nummer} to obtain
\begin{equation}\label{nummer2}
\begin{aligned}
&
  \EE\Bigl[
    \bigl|
      \cos\bigl(
        \widehat{X}_{ 2, n }^{ (\psi,n) }
        \cdot
        \psi\bigl(
          X_3^{ \psi }( T )
        \bigr)
      \bigr)
      -
      \cos\bigl(
        \widehat{X}_{ 2, n }^{ (\psi,n) }
        \cdot
        \psi\big(
          \widehat{X}_{ 3, n }^{ (n) }
        \big)
      \bigr)
    \bigr|^2
  \Bigr]
\\
& \qquad
  \leq 4 \cdot
  \EE\Bigl[
    \bigl|
      \sin\bigl(
        \tfrac{ 1 }{ 2 }
        \cdot
        \widehat{X}^{ (\psi,n) }_{ 2, n }
        \cdot
        \bigl[
          \psi\big(
            X^{ \psi }_3( T )
          \big)
          -
          \psi\big(
            \widehat{X}^{ (\psi, n) }_{ 3, n }
          \big)
        \bigr]
      \bigr)
    \big|^2
  \Bigr]
\\
&\qquad
  \leq
  \EE\Big[
    \big|
      \widehat{X}^{ (\psi,n) }_{ 2, n }
    \big|^2
    \cdot
    \big|
      \psi\big(
        X^{ \psi }_3( T )
      \big)
      -
      \psi\big(
        \widehat{X}^{ (\psi,n) }_{ 3, n }
      \big)
    \big|^2
    \cdot
    \mathbbm{1}_{
      \{
        \psi'(
          \max\{
            X^{ \psi }_3( T ) ,
            \widehat{X}^{ (\psi,n) }_{ 3, n }
          \}
        )
        \le
        n^{ \delta }
      \}
    }
  \Big]
\\ &
  \qquad\quad\quad
  +
  4 \cdot
  \PP\big(
    \psi'\big(
      \max\{
        X^{ \psi }_3( T ) ,
        \widehat{X}^{ (\psi,n) }_{ 3, n }
      \}
    \big)
    >
    n^{ \delta }
  \big)
\\
&\qquad
  \leq
  n^{ 2 \delta }
  \cdot
  \EE\Big[
    \big|
      \widehat{X}^{ (\psi,n) }_{ 2, n }
    \big|^2
  \Big]
  \cdot
  \EE\Big[
    \big|
      X^{ \psi }_3( T )
      -
      \widehat{X}^{ (\psi,n) }_{ 3, n }
    \big|^2
  \Big]
  \\
  & \qquad\quad\quad
  +
  4 \cdot
  \PP\big(
    \max\{
      X^{ \psi }_3( T ) ,
      \widehat{X}^{ (\psi,n) }_{ 3, n }
    \}
    >
    (\psi' )^{ - 1 }( n^{ \delta })
  \big)\\
& \qquad
  \leq
  \frac{ T^4 }{ 3 \, n^{ 2 ( 1 - \delta ) } }
  \cdot
    \sup_{ t \in [ 0 , \tau_1 ] }
    | f(t) |^2
  \cdot
    \sup_{ t \in [ \tau_2 , \tau_3 ] }
    | g'(t) |^2
\\
&\qquad\quad\quad
  +
  4 \cdot
  \PP\big(
    X^{ \psi }_3( T )
    >
          ( \psi' )^{ - 1 }( n^{ \delta })
  \big)
  +
  4 \cdot
  \PP\big(
    \widehat{X}^{ (\psi,n) }_{ 3, n }
    >
          ( \psi' )^{ - 1 }( n^{ \delta })
  \big)
  .
\end{aligned}
\end{equation}
Note that $\PP_{ X_{3}^{\psi}} = \mathcal N(0,\beta)$ and  $\PP_{\widehat X_{3,n}^{(\psi,n)}} = \mathcal N(0,\beta_n)$ and $\sup_{m\in\N}\beta_m \in [\beta,\infty)$. We may therefore apply Lemma \ref{lem:gaussian2} to conclude 
\begin{equation}\label{vv34}
\begin{aligned}
& \PP\big(
    X^{ \psi }_3( T )
    >
          ( \psi' )^{ - 1 }( n^{ \delta })
  \big)
  +
  \PP\big(
    \widehat{X}^{ (\psi,n) }_{ 3, n }
    >
          ( \psi' )^{ - 1 }( n^{ \delta })
  \big)\\
  & \qquad \le
  \exp\bigl(-\tfrac{|(\psi')^{-1}(n^\delta)|^2}{2\beta} \bigr)+
  \exp\bigl(-\tfrac{|(\psi')^{-1}(n^\delta)|^2}{2\sup_{m\in\N}\beta_m} \bigr).
\end{aligned}
\end{equation}
Combining \eqref{cv1}--\eqref{nummer} and \eqref{nummer2}--\eqref{eq:pol_estimate} 
ensures that there exist $c_1,c_2\in (0,\infty)$ 
such that for all $ n \in \{ n_0, n_0 + 1, \dots \} $ we have
\begin{equation}\label{neuenummer}
\begin{aligned}
\varepsilon_4\leq c_1\cdot \Bigl(\tfrac{1}{n^{2(1-\delta)}}+ \exp\bigl(-c_2
\cdot\bigl|\psi^{-1}(n^\delta)\bigr|^2\bigr)+ \exp\bigl(-c_2
\cdot\bigl|(\psi')^{-1}(n^\delta)\bigr|^2\bigr)\Bigr).
\end{aligned}
\end{equation}
By assumption we have
for all $ q \in (0,\infty) $
that
$
   \liminf_{ x \to \infty }
  \big[
    \psi( x ) \cdot
    \exp( - q x^2 )
  \big] = \infty
$. 
Hence, Lemma~\ref{lem:speed_psi}
ensures that
there exists
$ c_3 \in (0,\infty) $
such that for all $ n \in \N $ we have 
\begin{equation}
\label{eq:pol_estimate}
  \tfrac{ 1 }{
    n^{ ( 1 - \delta )
    }
  }
  \leq
  c_3
  \cdot
  \exp\bigl(
    - c_2 \left| \psi^{ - 1 }( n^{ \delta } ) \right|^2
  \bigr).
\end{equation}  

Combining \eqref{eq:second_component}, \eqref{eq:third_component}, \eqref{neuenummer}, 
and \eqref{eq:pol_estimate} finishes the proof.
\end{proof}

\begin{ex}\label{ex34}
Assume the setting in Section~\ref{sec:setting},
assume that $ \tau_1 < \tau_2 $, let
$ \beta \in (0,\infty) $ be given by 
$ \beta = \int_{ \tau_2 }^{ \tau_3 } \left| g(s) \right|^2 ds $,
let 
$ \psi_l \colon \R \to (0,\infty) $,
$ l \in \{ 1, 2 \} $, 
be the functions such that for all $ x \in \R $ 
we have
\begin{align}
\label{eq:example_psi1}
  \psi_1(x) &=
  \exp\!\left(
    x^3 + 2 x - ( 2 \beta )^{ 3 / 2 } - 2( 2 \beta )^{ 1 / 2 }
  \right)
  ,
\\
  \psi_2(x) &=
  \exp\!\left(x
    \exp\!\left( x^2 +1\right) -
    ( 2 \beta )^{ 1 / 2 }\exp( 2 \beta +1)
  \right)
  ,
\end{align}
and 
for every $ n \in \N $,
$ l \in \{ 1, 2 \} $
let 
$ 
  \widehat{X}^{ ( \psi_l, n ) } 
  \colon \{ 0, 1, \dots, n \} \times \Omega \to \R^4 
$
be the mapping such that
for all 
$ k \in \{ 0, 1, 2, \dots, n - 1 \} $
we have
$
  \widehat X^{(\psi_l,n)}_0 = 0
$ 
and 
\begin{equation}
  \widehat{X}^{ (\psi_l,n) }_{ k + 1 } =
  \widehat{X}^{ (\psi_l,n) }_k
  +
  \mu^{ \psi_l }( \widehat{X}^{ (\psi_l,n) }_k ) \, \tfrac{ T }{ n }
  +
  \sigma( \widehat{X}^{ (\psi_l,n) }_k ) \,
  \big(
    W( \tfrac{ ( k + 1 ) T }{ n } )
    -
    W( \tfrac{ k T }{ n } )
  \big)
  .
\end{equation}
Clearly,
we have
$ \psi_1, \psi_2 \in C^{ \infty }( \R , (0,\infty) ) $
and 
$ \psi_1\big( \sqrt{ 2 \beta } \big) = \psi_2\big( \sqrt{ 2 \beta } \big) = 1 $. 
Moreover, 
for all $ q \in (0,\infty) $
we have
\begin{equation}
  \liminf_{
    x \mapsto
    \infty
  }
  \big[
    \psi_1( x )
    \cdot
    \exp(
      - q x^2
    )
  \big]
  =
   \liminf_{
     x \mapsto
    \infty
  }
  \big[
    \psi_2( x )
    \cdot
    \exp(
      - q x^2
    )
  \big]
  = \infty. 
\end{equation}
Furthermore, for all $ x \in \R $ we have 
\begin{equation}
\begin{aligned}
  \psi_1'( x )
& =
  \left(
    3 x^2
    +
    2
  \right)\cdot
  \psi_1( x )
  > 0
  ,\\ 
  \psi_1''( x )
& =
  \bigl(
    6 x
    +
    (
      3 x^2 + 2
    )^2
  \bigr)\cdot
  \psi_1( x )
  =
  \bigl(
    9 x^4
    +
    3 x^2
    +
    3
    +(3 x +1)^2
      \bigr)\cdot
  \psi_1( x )
  > 0
\end{aligned}
\end{equation}
and
\begin{equation}
\begin{aligned} 
  \psi_2'( x )
 &=
  (2x^2+1)
  \exp\!\left( x^2+1
  \right)\cdot
  \psi_2( x )>0
  , 
\\
  \psi_2''( x )
& =
  \left(
    4x+(1+2x^2)\left(2x+(1+2x^2)\exp\!\left( x^2+1
  \right)\right)
    \right)
  \exp\!\left( x^2+1
  \right)
  \psi_2( x )
  \\
  & >
  \left(
    4x+(1+2x^2)\left(2x+2(1+2x^2)\right)
    \right)
  \exp\!\left( x^2+1
  \right)
  \psi_2( x )
  \\
  & \geq
  \left(
    4x+7/4\cdot (1+2x^2)
    \right)
  \exp\!\left( x^2+1
  \right)
  \psi_2( x ) \ge (17/28) \exp\!\left( x^2+1
  \right)
  \psi_2( x ) >0
  .
\end{aligned}
\end{equation}
Hence, 
$ \psi_1 $, $ \psi_1' $,  $ \psi_2 $, and $ \psi_2' $ are strictly increasing 
and we have $ \psi_1'( \R )=\psi_2'( \R ) = ( 0, \infty ) $.

Using Corollary \ref{cor1} and Theorem \ref{t2} with $ \delta = \nicefrac{ 1 }{ 2 } $ 
we conclude that there exist $ c_1, c_2 \in (0,\infty) $, $ n_0 \in \N $ 
such that for all $ k \in \{ 1, 2 \} $ and all $ n \in \{ n_0 , n_0 + 1 , \dots \} $ 
we have
\begin{equation}\label{first}
\begin{aligned}
&
 c_1
  \cdot
  \exp\bigl(
    -
    \tfrac{ 2 }{ \beta }
    \cdot
    \bigl|
      \psi_k^{ - 1 }\!\bigl( \tfrac{ 1 }{ c_1 }  \cdot n^3 \bigr)
    \bigr|^2
  \bigr)
\\ & \qquad \le 
\Big(
    \EE\Big[
    \bigl\|
      X^{ \psi_k }( T ) - \widehat{X}_{n}^{(\psi_k,n) }
    \bigr\|_{ \R^4 }^2
    \Big]
  \Big)^{ 1 / 2 }\\
  & \qquad\le   c_2
  \cdot
  \Big[
    \exp\bigl(
      - \tfrac{ 1 }{ c_2 } \cdot \bigl| \psi_k^{ - 1 }( n^{ 1/2 } ) \bigr|^2
    \bigr)
    +
    \exp\bigl(
      - \tfrac{ 1 }{ c_2 } \cdot
      \bigl| ( \psi_k' )^{ - 1 }( n^{1/2 } ) \bigr|^2
    \bigr)
  \Bigr]
  .
  \end{aligned}
\end{equation}

Next, we provide suitable minorants and majorants for the 
functions $ ( \psi_k )^{ - 1 } $, $ k \in \{ 1, 2 \} $, 
and 
$ ( \psi_k' )^{ - 1 } $, $ k \in \{ 1, 2 \} $. To this end 
we use the fact that for all $a\in\R$ and all strictly increasing continuous
functions $f_1, f_2\colon [a, \infty)\to\R$ with $f_1\geq f_2$
and $\lim_{x\to\infty} f_2(x)=\infty$  we have
\begin{equation}
\forall\, x\in[f_1(a),\infty)\colon\,\, x=f_2(f_2^{-1}(x))\leq f_1(f_2^{-1}(x))
\end{equation}
and therefore
\begin{equation}\label{inverse}
\forall\, x\in[f_1(a),\infty)\colon\,\,f_1^{-1}(x)\leq f_2^{-1}(x).
\end{equation}
Clearly, for all $ x \in [1,\infty) $ we have 
\begin{equation}
\begin{aligned}
\exp\!\left(
    x^3 +2- ( 2 \beta )^{ 3 / 2 } - 2( 2 \beta )^{ 1 / 2 }
  \right)
  \leq 
  \psi_1(x)
&
  \leq 
  \exp\!\left(
    3x^3
  \right),
\\
  \exp\!\left(
    \exp\!\left( x^2\right) -
    ( 2 \beta )^{ 1 / 2 }\exp( 2 \beta +1)
  \right)\leq \psi_2(x)&\leq \exp\!\left(
    \exp\!\left( x^2 +x+1\right) 
  \right)\leq \exp\!\left(
    \exp\!\left( 3x^2\right) 
  \right)
  ,
% \end{aligned}
% \end{equation}
% \begin{equation}
% \begin{aligned}
\\
  \psi_1'(x)
\leq 
  \exp\!\left(
    3x^2 +2
  \right) \cdot \psi_1(x)
& \leq 
  \exp\!\left(
    8x^3
  \right) ,
\\ 
  \psi_2'(x) 
\leq 
  \exp\!\left( 
    3 x^2 + 2
  \right)
  \cdot
  \psi_2( x )
& \leq 
  \exp\!\left(
    \exp\!\left( 8x^2 \right)
  \right)
  .
\end{aligned}
\end{equation}
We may therefore apply \eqref{inverse} with $a=1$ to obtain that for all 
$ x \in [ \exp(\exp(8)) , \infty ) $ we have
\begin{equation}\label{n45}
\begin{aligned}
\bigl(\ln (x)-2+( 2 \beta )^{ 3 / 2 } + 2( 2 \beta )^{ 1 / 2
}\bigr)^{1/3}\geq \psi_1^{-1}(x)&\geq 3^{-1/3}\cdot (\ln
(x))^{1/3},\\
(\ln(\ln
(x)+( 2 \beta )^{ 1 / 2 }\exp( 2 \beta +1)))^{1/2}\geq \psi_2^{-1}(x)&\geq 3^{-1/2}\cdot (\ln(\ln
(x)))^{1/2}
% \end{aligned}
% \end{equation}
% as well as
% \begin{equation}\label{n47}
% \begin{aligned}
\\
(\psi_1')^{-1}(x)&\geq 8^{-1/3}\cdot (\ln
(x))^{1/3},\\
(\psi_2')^{-1}(x)&\geq 8^{-1/2}\cdot (\ln(\ln
(x)))^{1/2}.
\end{aligned}
\end{equation}

Combining \eqref{first} with \eqref{n45} 
% and \eqref{n47} 
shows that there exist
$c_{1},c_{2},c_3,c_4\in (0,\infty)$, $n_0\in\N$ such that 
for all $ n \in \{ n_0, n_0 + 1, \dots \} $ we have
\begin{equation}
\begin{aligned}
  c_1 \cdot
  \exp\!\big(
    {
      - c_2
      \cdot 
      | \ln( n ) |^{ 2 / 3 }
    }
  \big)
& \leq
  \Big(
    \EE\Big[
    \bigl\|
      X^{ \psi_1 }( T ) - \widehat{X}_n^{ (\psi_1,n) }
    \bigr\|_{ \R^4 }^2
    \Big]
  \Big)^{ 1 / 2 }
\leq
  c_3 \cdot
  \exp\!\big(
    {
      - c_4 \cdot 
      | 
        \ln(n)
      |^{ 2 / 3 } 
    }
  \big),
\\
  c_1 \cdot 
  \exp\!\big(
    { - c_2 \cdot 
    \ln\!\big(
      \ln( n )
    \big)
    }
  \big)
& \leq
  \Big(
    \EE\Big[
    \bigl\|
      X^{ \psi_2 }( T ) - \widehat{X}_n^{ (\psi_2,n) }
    \bigr\|_{ \R^4 }^2
    \Big]
  \Big)^{ 1 / 2 }
\leq
  c_3 \cdot 
  \exp\!\big(
    {-
    c_4\cdot \ln\!\big( \ln( n ) \big) }
  \big)
  .
\end{aligned}
\end{equation}
In particular, in both cases the Euler-Maruyama scheme performs
asymptotically optimal on a logarithmic scale. 
\end{ex}

\section{Numerical experiments}
\label{sec:numerics}

We illustrate our theoretical findings by numerical simulations
of the mean error performance of the Euler scheme, the tamed Euler scheme, 
and the stopped tamed Euler scheme for a equation, which allows a decay of error 
not faster than 
$
  c
  \cdot
  \exp\!\big( - \nicefrac{ 1 }{ c } \cdot | \ln( n ) |^{ 2 / 3 } \big)
$ 
in terms of the number $ n \in \N $ of observations of the driving Brownian motion, 
where $c\in (0,\infty)$ is a real number which does not depend on $n \in \N $.

Assume the setting in Section~\ref{sec:setting},
assume that $ T = 1 $, 
$ \tau_1 = \tau_2 = \nicefrac{ 1 }{ 4 } $, 
$ \tau_3 = \nicefrac{ 3 }{ 4 } $, 
assume that for all $ x \in \R $ we have
\begin{equation}\label{4456}
\begin{aligned}
  f( x ) &  = 
  \1_{ ( - \infty , \nicefrac{ 1 }{ 4 } ) }(x)
  \cdot
  \exp\!\Big(
    3\ln(10)
    +
    \tfrac{ 1 }{  x - 1/4}
  \Big)
,\\
  g( x ) 
  & = 
  \1_{ ( \nicefrac{ 1 }{ 4 }, \nicefrac{ 3 }{ 4 } ) }(x)
  \cdot
  \exp\!\Big( \ln(2) +
    4\ln( 10)  + \tfrac{ 1 }{ 1/4 - x  } - \tfrac{ 1 }{  x - 3/4 }
  \Big)
, \\
  h( x ) 
 & =
  \1_{
    ( \nicefrac{ 3 }{ 4 } , \infty ) }( x )
    \cdot
  \exp\!\Big(
    4 \ln( 10 )+
    \tfrac{ 1 }{ 
       3 / 4- x 
    }
  \Big),
  \end{aligned}
  \end{equation}
(cf.\ Example~\ref{ex:fgh}),
let $\beta\in (0,\infty)$ be given by 
$ \beta = \int_{ 1/4 }^{ 3/4 } \left| g(s) \right|^2 ds $,
and let $ \psi \colon \R \to (0,\infty) $
be the function such that for all $ x \in \R $
we have
\[
  \psi(x) =
  \exp\!\big( 
    x^3 
  \big).
\]
Recall that the functions $f$, $g$, $h$, and $ \psi $ 
determine a drift coefficient $\mu^{\psi}\colon\R^4\to \R^4$ 
and a diffusion coefficient $\sigma\colon\R^4\to \R^4$, 
see~\eqref{coeff}.
Furthemore, recall that the fourth component of the solution $ X^\psi $ 
of the associated SDE at time $ 1 $ satisfies 
that it holds $ \PP $-a.s.\ that
\begin{equation}
  X^\psi_4( 1 ) 
  = 
  \smallint_{ 3 / 4 }^1 h(s) \, ds 
  \cdot 
  \cos\!\left(
    \smallint_0^{ 1 / 4 } f(s) \, dW(s)
    \cdot 
    \psi\!\left(
      \smallint\nolimits_{ 1 / 4 }^{ 3 / 4 } g(s) \, dW(s)
    \right)
  \right)
  ,
\end{equation}
see \eqref{sol}.

Furthermore, let
$
  \widehat{X}^{ (n), \eta }
  =
  \big(
    \widehat{X}^{ (n), \eta }_{ 1, ( \cdot ) }
    ,
    \widehat{X}^{ (n), \eta }_{ 2, ( \cdot ) }
    ,
    \widehat{X}^{ (n), \eta }_{ 3, ( \cdot ) }
    ,
    \widehat{X}^{ (n), \eta }_{ 4, ( \cdot ) }
  \big)
  \colon
  \{ 0, 1, \dots, n \} \times \Omega \to \R^4
$,
$ n \in \N $,
$
  \eta \in \{ 1, 2, 3 \}
$,
be the mappings such that for all 
$ \eta \in \{ 1, 2, 3 \} $,
$ n \in \N $,
$ k \in \{ 0, 1, \dots, n - 1 \} $
we have
$
  \widehat{X}^{ (n), \eta }_0 = 0
$
and
\begin{equation}
\begin{aligned}
&
  \widehat{X}^{ (n), 1}_{ k + 1 } 
  = 
  \widehat{X}^{ (n), 1 }_k
  +
  \mu^{ \psi }( 
    \widehat{X}^{ (n), 1 }_k 
  ) 
  \, \tfrac{ 1 }{ n }
  +
  \sigma( 
    \widehat{X}^{ (n), 1 }_k 
  ) \, 
  \big(
    W( \tfrac{ k + 1  }{ n } )
    -
    W( \tfrac{ k  }{ n } )
  \big)
  ,
\\
&
  \widehat{X}^{ (n), 2}_{ k + 1 } 
  =
  \widehat{X}^{ (n), 2 }_k
  +
  \frac{
    \mu^{ \psi }( 
      \widehat{X}^{ (n), 2 }_k 
    ) 
    \, \frac{ 1 }{ n }
  }{
      1 
      + 
      \|
        \mu^{ \psi }( 
          \widehat{X}^{ (n), 2}_k 
        ) 
      \|_{ \R^4 }
      \,
      \frac{ 1 }{ n }
  }
  +
  \sigma( 
    \widehat{X}^{ (n), 2 }_k 
  ) \, 
  \big(
    W( \tfrac{  k + 1  }{ n } )
    -
    W( \tfrac{ k }{ n } )
  \big)
  ,
\\
&
  \widehat{X}^{ (n), 3 }_{ k + 1 } 
  = 
  \widehat{X}^{ (n), 3 }_k
\\ 
&
  +
  \mathbbm{1}_{
    \left\{ 
      \| 
        \widehat{X}^{ (n), 3}_k 
      \|_{\R^4}
      \leq 
      \exp\left( 
        | 
          \ln( n  )
        |^{ 1 / 2 }
      \right)
    \right\}
  } 
  \left[
    \frac{
      \mu^{ \psi }( 
        \widehat{X}^{ (n), 3}_k 
      ) 
      \, \tfrac{ 1 }{ n }
      +
      \sigma( 
        \widehat{X}^{ (n), 3 }_k 
      ) \, 
      \big(
        W( \frac{  k + 1  }{ n } )
        -
        W( \frac{ k  }{ n } )
      \big)
    }{
      1 + 
      \big\|
        \mu^{ \psi }( 
          \widehat{X}^{ (n), 3 }_k 
        ) 
        \, \tfrac{ 1 }{ n }
        +
        \sigma( 
          \widehat{X}^{ (n), 3 }_k 
        ) \, 
        \big(
          W( \frac{  k + 1  }{ n } )
          -
          W( \frac{ k  }{ n } )
        \big)
      \big\|^2_{ \R^4 }
    }
  \right]
  .
\end{aligned}
\end{equation}
Thus 
$ \widehat{X}^{ (n), 1} $, $ \widehat{X}^{ (n), 2 } $, $ \widehat{X}^{ (n), 3} $ 
are the Euler scheme (see Maruyama~\cite{m55}), 
the tamed Euler scheme in 
Hutzenthaler et al.~\cite{HutzenthalerJentzenKloeden2012}, 
and the stopped tamed Euler scheme
in Hutzenthaler et al.~\cite{HutzenthalerJentzenWang2013}, respectively, 
each with time-step size $ 1 / n $. 

Let 
$
  \varepsilon_n^{ \eta }
  \in [0,\infty)
$,
$ n \in \N $,
$ \eta \in \{ 1, 2, 3 \} $,
be the real numbers with the property that
for all $ n \in \N $,
$ \eta \in \{ 1, 2, 3 \} $
we have
\[
  \varepsilon_n^\eta 
  =
  \EE\bigl[
    | X_4^\psi(1) - \widehat{X}_{4,n}^{ (n), \eta} | 
  \bigr], 
\]
let 
$
  \bar{f} \colon \R \to \R
$
and 
$
  \bar{\psi} \colon \R \to (0,\infty)
$
be the functions such that for all 
$ x \in \R $ we have
$
  \bar{f}( x ) = 
  \exp( 
    ( 2 \beta )^{ 3 / 2 } 
  ) 
  \cdot f(x)
$
and
$
  \bar{\psi}( x ) = 
  \exp( 
    - ( 2 \beta )^{ 3 / 2 } 
  )
  \cdot \psi(x)
$,
and let $ \alpha_1, \alpha_2, \alpha_3, \bar{c}, \bar{C} \in (0,\infty) $ 
be the real numbers given by 
\begin{gather}
  \alpha_1
   =
  \smallint_0^{ \tau_1 }
  | \bar{f}(s) |^2 \, ds ,
\qquad
  \alpha_2
 =
  \sup_{ s \in [ 0 , \nicefrac{ \tau_1 }{ 2 } ] }
  | \bar{f}'(s) |^2
  ,
\qquad
  \alpha_3
  =
  \inf_{ s \in [ 0 , \nicefrac{ \tau_1 }{ 2 } ] }
  | \bar{f}'(s) |^2
  ,
\\
 \bar{c} =
  \frac{
    | 
      \int_{ \tau_3 }^1
      h(s) \, ds
    |
  }{
    8
    \,
    \pi^{ 3 / 2 }
    \exp( \tfrac{ \pi^2 }{ 4 } )
  }
  ,
\qquad
  \bar{C} =
  \frac{
    \sqrt{ 12 }
    \,
    \max\{
      1 ,
      \sqrt{\alpha_2}
    \}
  }{
    \sqrt{ \alpha_3 }
    \min\{
      1 ,
      \sqrt{ \tfrac{ \alpha_1 }{ 2 } }
    \}
  }
  .
\end{gather}
In the next step we note that
$ \bar{\psi} \in C^{ \infty }( \R, (0,\infty) ) $
is strictly increasing,
we note that
$
  \liminf_{ x \to \infty } 
  \bar{\psi}( x ) = \infty
$,
and we note that
$ \bar{\psi}( \sqrt{ 2 \beta } ) = 1 $.
We can thus apply 
inequality~\eqref{eq:cor1_eq2}
in Corollary~\ref{cor1} 
(with the functions 
$
  \bar{f} 
$,
$ 
  g 
$,
$
  h 
$,
and  
$
  \bar{\psi} 
$)
to obtain that
for all
$
  n \in \N
$,
$ s_1, \dots, s_n \in [ 0, 1 ] $
and all measurable
$
  u \colon
  \R^n \to \R
$
we have 
$ 
  [ 8 \bar{C} n^{ 3 / 2 } ( \tau_1 )^{ - 3 / 2 } , \infty) \subset \bar{\psi}( \R ) 
$ 
and
\begin{equation}
  \EE\Big[
    \bigl|
      X^{ \psi }_4(1)
      -
      u\big(
        W( s_1 ),
        \dots ,
        W( s_n )
      \big)
    \bigr|
  \Big]
  \geq
  \bar{c}
  \cdot
  \exp\Bigl(
    -
    \tfrac{ 2 }{ \beta }
    \cdot
    \bigl|
      \bar{\psi}^{ - 1 }\bigl( \tfrac{ 8 \, \bar{C} }{ ( \tau_1 )^{ 3 / 2 } } \cdot 
      n^{ 3 / 2 } \bigr)
    \bigr|^2
  \Bigr)
  .
\end{equation}
This and the fact that 
$
  \forall \, y \in \bar{\psi}( \R ) \colon
  \bar{\psi}^{ - 1 }( y ) = 
  \big[
    \ln(
      y \cdot \exp( ( 2 \beta )^{ 3 / 2 } )
    )
  \big]^{ 1 / 3 }
$
ensure that
for all
$
  n \in \N
$,
$ s_1, \dots, s_n \in [ 0, 1] $
and all measurable
$
  u \colon
  \R^n \to \R
$
we have 
\begin{equation}
\begin{aligned}
&
  \EE\Big[
    \bigl|
      X^{ \psi }_4(1)
      -
      u\big(
        W( s_1 ),
        \dots ,
        W( s_n )
      \big)
    \bigr|
  \Big]\\
  & \qquad\qquad
  \geq
  \bar{c}
  \cdot
  \exp\Bigl(
    -
    \tfrac{ 2 }{ \beta }
    \cdot
    \bigl|
      \ln\bigl( 
        \tfrac{ 8 \, \bar{C} }{ ( \tau_1 )^{ 3 / 2 } } 
        \cdot 
        \exp( ( 2 \beta )^{ 3 / 2 } )
        \cdot 
        n^{ 3 / 2 } 
      \bigr)
    \bigr|^{ 2 / 3 }
  \Bigr)
\\ &\qquad\qquad =
  \bar{c}
  \cdot
  \exp\Bigl(
    -
    \tfrac{ 2 }{ \beta }
    \,
    \bigl|
      \ln\bigl( 
        \tfrac{ 
          8 \, \bar{C} 
          \exp( ( 2 \beta )^{ 3 / 2 } )
        }{ ( \tau_1 )^{ 3 / 2 } } 
      \bigr)
      +
      \tfrac{ 3 }{ 2 }
      \ln( n ) 
    \bigr|^{ 2 / 3 }
  \Bigr)
\\ & \qquad\qquad\geq
  \bar{c}
  \cdot
  \exp\Bigl(
    -
    \tfrac{ 2 }{ \beta }
    \,
    \bigl|
      \ln\bigl( 
        \tfrac{ 
          8 \, \bar{C} 
          \exp( ( 2 \beta )^{ 3 / 2 } )
        }{ ( \tau_1 )^{ 3 / 2 } } 
      \bigr)
    \bigr|^{ 2 / 3 }
  \Bigr)
  \cdot
  \exp\Bigl(
    \tfrac{ - 2^{ 1 / 3 } \, 3^{ 2 / 3 } }{ \beta }
    \cdot
    |
      \ln( n ) 
    |^{ 2 / 3 }
  \Bigr)
  .
\end{aligned}
\end{equation}
In particular, this proves that there exists 
a real number $ c \in (0,\infty) $
such that for all 
$ \eta \in \{ 1, 2, 3 \} $,
$ n \in \N $
we have
\begin{equation}
  \varepsilon_n^\eta 
=
  \EE\bigl[
    | X_4^\psi(1) - \widehat{X}_{4,n}^{ (n), \eta} | 
  \bigr] 
  \ge c\cdot 
  \exp\!\big(
    - \tfrac{ 1 }{ c } \cdot | \ln(n) |^{ 2 / 3 }
  \big) .
\end{equation}

In the next step 
let $ m = 5000 $, $ N = 2^{ 21 } $, 
let $ B = ( B_1, \dots, B_m ) \colon [0,1] \times \Omega \to \R^m $ be an $ m $-dimensional 
standard Brownian motion,
and let
$ Y^N = ( Y^N_1, \dots, Y^N_m ) \colon \Omega \to \R $,
$ N \in \N $,
be the random variables with the property that
for all $ N \in \N $, $ k \in \{ 1, 2, \dots, m \} $ we have
\begin{equation}
\label{reference}
\begin{aligned}
  Y^N_k & =  
  \int_{ 3 / 4 }^1 h(s) \, ds
  \cdot 
  \cos\!\left(
    \textstyle
    \frac{ 1 }{ N }
    \sum\limits_{ i = 0 }^{ \lfloor N / 4 \rfloor } 
    f'( \tfrac{ i }{ N }) \cdot 
    B_k( \tfrac{ i }{ N } ) 
    \cdot 
    \psi\!\left(
      - \frac{ 1 }{ N }
      \sum\limits_{ i = \lceil N / 4 \rceil }^{ \lfloor 3 N / 4 \rfloor } 
      g'( \tfrac{ i }{ N }) 
      \,
      B_k( 
        \tfrac{ i }{ N }
      )
    \right)
    \displaystyle
  \right)
  .
\end{aligned}
\end{equation}
The random variables
$ Y^N_k $, $ k \in \{ 1, 2, \dots, m \} $, $ N \in \N $,
are used to get reference estimates 
of realizations of $ X^{ \psi }_4( 1 ) $. 
Our numerical results are based on a simulation 
\begin{equation}
\label{sim}
  ( b_1, \dots, b_m ) 
  = 
  \big(
    ( b_{ 1, i } )_{ i \in \{ 0, 1, \dots, N \} },
    \dots, 
    ( b_{ m, i } )_{ i \in \{ 0, 1, \dots, N \} }
  \big)
  \in \R^{ (N+1) m } 
\end{equation}
of a realization of
$
  \big( 
    ( 
      B_1( \nicefrac{ i }{ N } ) 
    )_{ i \in \{ 0, 1, \dots, N \} } , 
    \dots , 
    (
      B_m( \nicefrac{ i }{ N } ) 
    )_{ i \in \{ 0, 1, \dots, N \} } 
  \big)
$
(a realization of
$
  ( B_1, \dots, B_m ) 
$ 
evaluated at the equidistant times 
$ \nicefrac{ i }{ N } $, $ i \in \{ 0, 1, \dots, N \} $). 
Based on $ ( b_1, \dots, b_m ) $ 
we compute a simulation 
$ ( y_1, \dots, y_m ) \in \R^m $ 
of a realization of 
$
  ( Y^N_1, \dots, Y^N_m )
$
and 
based on $ ( b_1, \dots, b_m ) $ 
we compute 
for every 
$
  \eta \in \{ 1, 2, 3 \} 
$ 
and every 
$
  n \in \{ 2^0, 2^1, \dots, 2^{ 19 } \} 
$ 
a simulation $(x^{(n),\eta}_{1},\dots,x^{(n),\eta}_{m})\in \R^m$ of a corresponding realization of $m$ independent copies of $\widehat{X}_{4,n}^{ (n), \eta}$. 
Then
for every 
$
  \eta \in \{ 1, 2, 3 \} 
$ 
and every 
$
  n \in \{ 2^0, 2^1, \dots, 2^{ 19 } \} 
$ 
the real number
\begin{equation}\label{errest}
  \widehat \varepsilon^\eta_n = \frac{1}{m}\sum_{\ell=1}^m |y_\ell - x_\ell^{(n),\eta}|
\end{equation}
serves as an estimate of 
$
  \varepsilon_n^\eta
  =
  \EE\bigl[
    | X_4^\psi(1) - \widehat{X}_{4,n}^{ (n), \eta} | 
  \bigr] 
$.

Figure \ref{fig1} shows, on a log-log scale, 
the plots of the error estimates $\widehat \varepsilon^1_n$, $\widehat \varepsilon^2_n$, $\widehat \varepsilon^3_n$ 
versus the number of time-steps 
$ n \in \{ 2^0, 2^1, 2^2, \dots, 2^{ 18 } , 2^{ 19 } \} $. 
Additionally, the powers 
$ n^{ - 0.01 } $, 
$ n^{ - 0.05 } $, 
$ n^{ - 0.1 } $,  
$ n^{ - 0.2 } $ 
are plotted versus 
$ n \in \{ 2^0, 2^1, 2^2, \dots, 2^{ 18 }, 2^{ 19 } \} $. 
The results provide some numerical evidence for the theoretical 
findings in Corollary~\ref{cor1b}, that is, none of the three schemes 
converges with a positive polynomial strong order of convergence to 
the solution at the final time.

\begin{figure}%[!h]
\centering
\hspace*{2cm}
\includegraphics[trim=1cm 10cm .5cm 8cm,clip, width=19cm]
{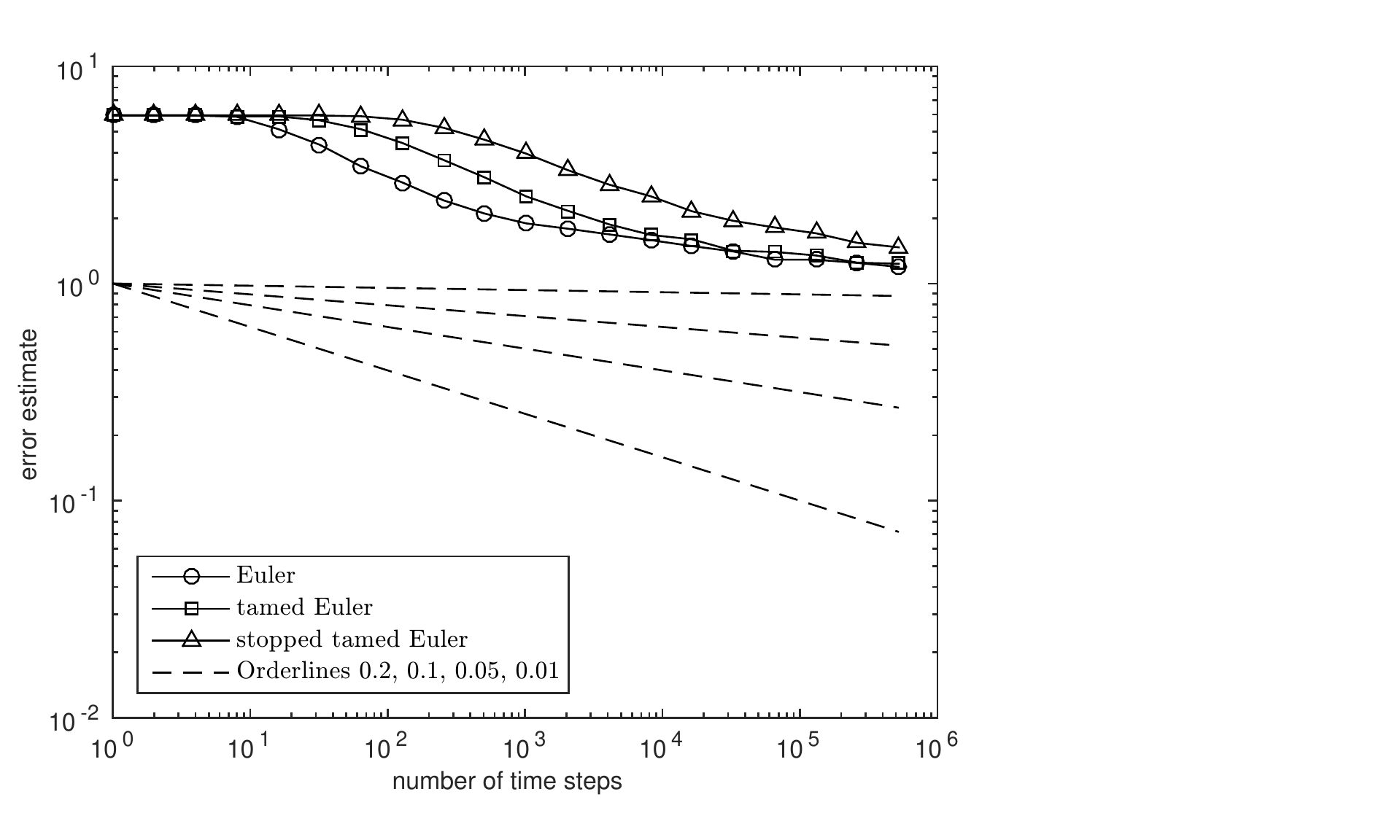}
\caption{Error vs. number of time steps}\label{fig1}
\end{figure}

\bibliographystyle{acm}
\bibliography{bibfile}

\end{document}